\documentclass{amsart}
\usepackage{amssymb,mathrsfs}
\usepackage{color,umoline}
\usepackage[dvipsnames]{xcolor}
\usepackage{graphicx}
\usepackage{enumitem}
\usepackage[font=small,labelfont=bf]{caption}
\usepackage{tikz, caption}
\usetikzlibrary{patterns}
\usepackage{relsize}
\usepackage{kotex}
\usepackage[mathscr]{eucal}
\usepackage{dsfont}
\usepackage{verbatim}
\usepackage[draft]{todonotes}
\usepackage[
citecolor=black,  
hypertexnames=false]{hyperref}

\usepackage[dvipsnames]{xcolor}
\usepackage{subfig}

\DeclareFontFamily{U}{mathx}{}
\DeclareFontShape{U}{mathx}{m}{n}{<-> mathx10}{}
\DeclareSymbolFont{mathx}{U}{mathx}{m}{n}
\DeclareMathAccent{\widecheck}{0}{mathx}{"71}

\def\wh#1{\widehat{#1}}

\newcommand{\C}{\mathbb{C}}
\newcommand{\N}{\mathbb{N}}
\newcommand{\R}{\mathbb{R}}
\newcommand{\Z}{\mathbb{Z}}

\newcommand{\supp}{\operatorname{supp}}
\newcommand{\dist}{\operatorname{dist}}

\newtheorem{thm}{Theorem}[section]

\newtheorem{prop}[thm]{Proposition}
\newtheorem{coro}[thm]{Corollary}
\newtheorem{lemma}[thm]{Lemma}
\numberwithin{equation}{section}
\newtheorem{rem}{Remark}

\newcommand\CL{{\mathcal L}}
\newcommand\CP{{\mathcal P}}

\newcommand{\ep}{{\varepsilon}}

\newcommand{\inp}[2]{\langle #1, #2\rangle}

\newcommand{\Be}{\begin{equation}}
\newcommand{\Ee}{\end{equation}}

\begin{document}

\author[Jeong]{Eunhee Jeong}
\address[Jeong]{Department of Mathematics Education, and  Institute of Pure and Applied Mathematics, Jeonbuk National University, Jeonju 54896, Republic of Korea}
\email{eunhee@jbnu.ac.kr}

\author[Lee]{Sanghyuk Lee}
\address[Lee]{Department of Mathematical Sciences and RIM, Seoul National University, Seoul 08826, Republic of  Korea}
\email{shklee@snu.ac.kr}

\author[Ryu]{Jaehyeon Ryu}
\address[Ryu]{School of Mathematics, Korea Institute for Advanced Study, Seoul
02455, Republic of Korea} 
\email{jhryu@kias.re.kr}


\title[Almost everywhere convergence]{Almost everywhere convergence of Bochner--Riesz means for the 
twisted laplacian}
\begin{abstract} 
Let $\mathcal L$ denote the twisted Laplacian in $\C^d$.   
We study almost everywhere convergence of the Bochner--Riesz mean $S^\delta_{t}(\mathcal L) f$ of $f\in L^p(\mathbb C^d)$   as $t\to \infty$, which is an expansion  
of $f$ in the special Hermite functions.  For $2\le p\le \infty$, we obtain the sharp range of the summability indices $\delta$ for which the convergence of $S^\delta_{t}(\mathcal L) f$ holds for all $f\in L^p(\C^d)$.
\end{abstract}

\maketitle

\section{Introduction}
Almost everywhere convergence of the Bochner--Riesz mean 
\[
S_t^\delta f(x) = \frac{1}{(2\pi)^d} \int_{\mathbb R^d} e^{i\inp{x}{\xi}} \Big(1-\frac{|\xi|^2}{t^2}\Big)_+^\delta \widehat{f}(\xi) d\xi, \quad \delta \ge 0
\]
as $t\to \infty$ has been an important topic in classical harmonic analysis.  
In analogue to the Bochner--Riesz conjecture which concerns $L^p$ convergence of $S_t^\delta f$, the problem of determining the optimal summability index $\delta$ (depending on $p$) for which $S_t^\delta f\to f$ almost everywhere (abbreviated to a.e. in what follows) for every $f\in L^p(\R^d)$ has been extensively studied by various authors   (\cite{St58, Ca83, CRV88, Ch85}). In particular, 
 for $2\le p\le \infty$, this problem was essentially settled by Carbery--Rubio de Francia--Vega \cite{CRV88}. They proved that a.e. convergence holds for any $f\in L^p(\R^d)$ if 
\begin{align}\label{c:deltacbr}
\delta > \delta(p,d) := \max{\Big(0,  d\Big(\frac12-\frac1p\Big)-\frac12\Big)}
\end{align}
 for $2\le p < \infty$.  Discussions on the necessity of the condition \eqref{c:deltacbr} can be found in \cite{CaS88, LS15}. 
However, as for the case $1\le p <2$,  the pointwise behavior of  the Bochner--Riesz mean of $L^p$ functions  turned out to be quite different. Not much is known beyond the classical result due to Stein \cite{St58}. 
We refer to \cite{Ta98, Ta02, LW20} and references therein for the recent results.

Via  spectral decomposition,  Bochner--Riesz means can be defined for a general positive self-adjoint operator 
which admits a spectral decomposition $L=\int_0^\infty \lambda dE_L(\lambda)$   in $L^2$, where  $dE_L$ denotes the spectral measure associated with $L$. In fact, 
the Bochner--Riesz means associated with $L$ are given by 
\[S_t^\delta(L) f= \int_0^{t} \Big(1-\frac{\lambda}{t^2}\Big)^\delta dE_L(\lambda) f.\] 
More generally, for any measurable  function $m$   the operator $m(L)$ is defined by $m(L)f=\int  m(\lambda) dE_L(\lambda) f$. 
There is a large body of literature concerning generalizations  of the Bochner--Riesz means $S_t^\delta$ to various operators $L$. See, for example,  \cite{CLSY20, CS88, GM02, GHS13, HM21, Ka94, So87, So02, SZ98, Tay89, Th98a}
and references therein.

In this paper we are concerned with  Bochner--Riesz means associated with the twisted Laplacian  $\mathcal L$ on $\C^d\cong\R^{2d}$, which 
is a self-adjoint second-order differential operator defined by
\[
\mathcal L= - \sum_{j=1}^d\bigg(\Big(\frac{\partial}{\partial x_j}-\frac i2y_j\Big)^2 + \Big(\frac{\partial}{\partial y_j}+\frac i2x_j\Big)^2\bigg), \qquad (x,y)\in \R^d\times \R^d.
\]
The twisted Laplacian $\mathcal L$ is of particular interest  in mathematical physics and quantum physics since it is a typical example of the Schr\"odinger operators with constant magnetic fields.
Also, $\mathcal L$ has a close connection to the sub-Laplacian on the Heisenberg group, which is a unique operator on the Heisenberg group which is homogeneous of degree two and invariant under the left action and rotation.  For more about those, we refer to \cite{Th98b, RS75}.

\subsection*{Bochner--Riesz means associated with $\mathcal L$.} 
Set $\N_0=\N\cup\{0\}.$ 
For $\alpha,\beta\in \N_0^d$ the special Hermite function $\Phi_{\alpha,\beta}$ is given by 
\[
\Phi_{\alpha,\beta}(z):= (2\pi)^{-\frac d2}\int_{\R^d} e^{i\inp x\xi}\Phi_\alpha\Big(\xi +\frac12 y\Big)\Phi_\beta\Big(\xi-\frac 12 y\Big)d\xi,\quad z=x+iy\in\mathbb C^d,
\]
where  $\Phi_\alpha$ denotes  the normalized Hermite functions  on $\R^d$.  
$\Phi_{\alpha,\beta}$ is an eigenfunction of $\mathcal L$ with eigenvalue $(2|\beta|+d)$, i.e., $\mathcal L\Phi_{\alpha,\beta} = (2|\beta|+d)\Phi_{\alpha,\beta}$. Here $|\beta| = \sum_i\beta_i$. 
Furthermore,  $\{ \Phi_{\alpha, \beta}\}$ forms  an orthonormal basis for $L^2(\mathbb C^d)$ and  the spectrum of $\mathcal L$ is 
$
2\N_0+d := \{2k + d : k\in \N_0\}
$ (see \cite{Th93}). 
For   $\mu\in 2\N_0+d$, let $\mathcal P_\mu$ denote the spectral projection operator defined by
\[
\mathcal P_\mu f(z) = \sum_{\beta : 2|\beta|+d = \mu} \sum_{\alpha\in \N_0^d} \inp{f}{\Phi_{\alpha,\beta}}\Phi_{\alpha,\beta}(z).
\]
The Bochner--Riesz means $S_{t}^\delta(\mathcal L)$ for the twisted Laplacian $\mathcal L$ is given by
\[
S_t^\delta(\mathcal L) f(z) = \sum_{\mu\in 2\N_0+d: \, \mu\le t^2} \Big(1-\frac{\mu}{t^2}\Big)^\delta \mathcal P_{\mu} f(z).
\]
$L^p$ convergence of $S_t^\delta(\mathcal L) f$ was studied in local and global settings by several authors \cite{Th91,Th98a, SZ98,LR22} (see, also, \cite{KR, JLR22}).
In particular, it was shown in \cite{SZ98} that $S_t^\delta(\mathcal L) f$ converges in $L^p$  for  $2(2d+1)/(2d-1)<p<\infty$ if 
\[
\delta > \delta(p,2d).
\]
The range of $\delta$ is sharp as can be seen from the transference theorem due to Kenig--Stanton--Tomas \cite{KST}. 
The range  of  $p$ for which $L^p$ convergence holds with the sharp summability index was further extended in a local setting \cite{LR22}.

In this paper,  motivated by the recent work of Chen--Duong--He--Lee--Yan \cite{CDHLY21} on  a.e. convergence of  Bochner--Riesz means for the Hermite operator 
$\mathcal H:=-\Delta+ |x|^2$,  we study  a.e. convergence of $S_t^\delta(\mathcal L) f$, that is to say, characterizing $\delta=\delta(p)$ for which  
\begin{align}\label{id:aeconv}
\lim_{t\to\infty} S_t^\delta(\mathcal L) f(z) = f(z) \quad \text{a.e.}  \quad  \forall  f\in L^p(\C^d)
\end{align}
for $2\le p\le \infty$. Compared with $L^p$ convergence of $S_t^\delta(\mathcal L) f$, its a.e. convergence has not been well studied.  It was shown only for relatively large summability indices. 
In  \cite{Th93}, Thangavelu showed that  \eqref{id:aeconv} holds for $1\le p\le \infty$ if $\delta>d-1/3$  and  
for $p>4/3$ if $\delta>d-1/2$.

The next is our first  result, which  provides a complete picture of $p$ and $\delta$ for \eqref{id:aeconv} to hold except for some endpoint cases when  $2\le p\le \infty$.

\begin{thm}\label{thm:ptconv}
    Let $2\le p\le \infty$, $\delta\ge 0$, and $d\ge 1$. If 
   \[
\delta > \delta(p,2d)/2
\]
     then \eqref{id:aeconv} holds true. 
    Conversely, \eqref{id:aeconv} fails if $\delta<\delta(p,2d)/2$.
\end{thm}

It should be pointed out that the critical summability index for a.e. convergence is only half of that for $L^p$ convergence. 
A similar result was  obtained in \cite{CDHLY21} for the Bochner--Riesz means $S_t^\delta(\mathcal H)f$
associated with  the Hermite operator. 
In fact, for $2\le p < \infty$  it was shown that 
$S_t^\delta(\mathcal H)f$ converges to $f$ a.e. as $t\to \infty$   for all $f\in L^p(\R^d)$ provided that
 $\delta > \delta(p,d)/2$. 
As to be discussed later, this kind of improvement of summability index is related to the facts that  $\mathcal L$  and $\mathcal H$ have discrete spectrums  bounded away from the zero and
the kernels of  the multiplier operators $\eta((\mu-\mathcal L)/R)$ and $\eta((\mu-\mathcal H)/R)$ are essentially supported  near the diagonal $\{(z,z')\in\C^d\times \C^d : z=z'\}$ 
(see Lemma \ref{lem:etapt} and \ref{lem:etapth}). 

Theorem \ref{thm:ptconv} includes the case $p=\infty$, which was not covered in the previous works (\cite{CRV88, CDHLY21}). 
In particular, this is possible because our approach does not rely on the fact that the weight  $\Psi_\alpha$ is in $A_2$ class (see  the discussion below Theorem \ref{thm:maxi} for more detail). 
In \cite{CDHLY21},  the sharpness of summability index was shown by making use of the Nikishin--Maurey theorem. 
However,  we verify  the necessity part of  Theorem \ref{thm:ptconv}  by directly constructing $L^p$ functions for which a.e. convergence fails if $\delta<\delta(p,2d)/2$. 
More precisely,  for $4d/(2d-1)< p\le \infty$, we shall show that there exists a function $f\in L^p(\C^{d})$ such that
\begin{equation}\label{eq:mea} 
|\{z\in\C^{d} : \sup_{t > 0}|S_{t}^\delta(\mathcal L) f(z)| = \infty\}|\gtrsim 1
\end{equation}
if $\delta<\delta(p,2d)/2$.    
See Section \ref{sec:nece} for the detail. In particular, this enables us to show  
sharpness of summability index for  $f\in L^\infty$, which is not allowed when  using the Nikishin--Maurey theorem.

\subsection*{Maximal estimate on a weighted $L^2$ space}
To prove the sufficiency part of Theorem \ref{thm:ptconv}, we consider the maximal Bochner--Riesz operator
\[
S_*^\delta(\mathcal L) f(z) = \sup_{t>0} |S_t^\delta(\mathcal L)f(z)|. 
\]
$L^p$ boundedness of the maximal operator $f\to \sup_{t>0} |S_t^\delta f|$  of the classical Bochner--Riesz means  has been studied  to show a.e.   
convergence of $S_t^\delta f$ 
(see \cite{ St58, Ca83, Le04, LRS12, Lee18, gjw, gow} and references therein).  Rather than showing $L^p$  boundedness of  $S_*^\delta(\mathcal L)$, we take an approach introduced in \cite{CRV88}  (\cite{An17, GM02, HM21, LS15}) which relies on a weighted $L^2$ estimate. 
For the purpose we consider a weight  function
\[
\textstyle \Psi_\alpha(z) = \sum_{j\ge 0} 2^{-\alpha j}\chi_{\mathbb A_j}(z),   \quad \alpha\in \mathbb R, 
\]
where 
$
\mathbb A_j= \{z\in \C^{d} : 2^{j-1}<|z|\le 2^j\}
$
for $j\ge 1$ and $\mathbb A_0 = \{z\in \C^{d} : |z|\le 1\}$. Note that $\Psi_\alpha(z)\sim (1+|z|)^{-\alpha}$. 
Theorem \ref{thm:ptconv} is a consequence of the following.

\begin{thm}\label{thm:maxi}
Let $\alpha \ge  0$. If $\delta>\max\{(\alpha-1)/4, 0\}$, then we have
\begin{align}\label{e:maxi}
\|S_*^\delta(\mathcal L) f \|_{L^2(\C^d, \Psi_\alpha)} \le C \|f\|_{L^2(\C^d, \Psi_\alpha)}
\end{align}
for a constant $C>0$.
\end{thm}

Theorem \ref{thm:maxi}  is sharp in that \eqref{e:maxi} fails if $\delta< (\alpha-1)/4$ (see Remark \ref{l2w} below). By a standard argument (see  a discussion below Corollary \ref{thm:aepoly}), the sufficiency  part of Theorem \ref{thm:ptconv} follows from  Theorem \ref{thm:maxi}. 
In the previous works (\cite{CRV88, CDHLY21}), Littlewood--Paley inequality and the fact that  the weights are contained in $A_2$-class  played a role in proving the weighted $L^2$ inequality. This in turn results in  imposing a bound on the growth order of the weights, that  is to say, $\alpha<d$. However,   our result continues to be valid  without an upper bound on $\alpha$.
This allows us to extend Theorem \ref{thm:ptconv} to a class of functions which have growth at infinity. 

\begin{coro}\label{thm:aepoly}
Let  $d\ge 1$ and $\beta\ge 0$. Set
$ \gamma(p,d,\beta)=\max(0, \beta+d(\frac12-\frac1p)-\frac12)$.
 If $\delta > \gamma(p, 2d,\beta)/2$, then  
$
\lim_{t\to\infty} S_t^\delta(\mathcal L) f= f$ {a.e.} 
 whenever $\Psi_\beta f \in L^p(\C^d)$. Conversely, if \eqref{id:aeconv} holds for all $f$ satisfying $ \Psi_\beta f\in L^p(\C^d)$ for some $p\in (4d/(2d-1+2\beta),  \infty]$, then $\delta \ge \gamma(p, 2d,\beta)/2$.
\end{coro}

The sufficiency part of the corollary  is a simple consequence of the embedding $L^p(\C^d)\hookrightarrow L^2(\C^d, \Psi_\alpha)$  when $\alpha>  2d(1- 2/p)$. 
In particular, note that  $\alpha-1  >  2\delta(p, 2d)$ for $p\ge 4d/(2d-1)$.

\subsubsection*{Our approach} 
As in \cite{CDHLY21}, we exploit the special spectral properties of the twisted Laplacian, that is to say, a generalized  trace lemma (Lemma \ref{lem:trace}) and the fact that 
the discrete spectrum of $\mathcal L$ is bounded away from the origin.  However, there are significant differences between the problems for $\mathcal H$ and $\mathcal L$. For example,  a crucial inequality  which relates the weight function and $\mathcal H$ (\cite[Lemma 1.4]{CDHLY21}) does not generally hold for $\mathcal L$.    Instead of following   \cite{CDHLY21}, we devise a simpler and more direct approach which relies on estimates for the kernel of associated spectral multipliers (Lemma \ref{lem:etapt}). Most of all, our approach  does not rely  on the Littlewood--Paley  and  $A_p$ weight theories. 
(See Section \ref{sec:max} and \ref{sec:sq}.)  Besides,  we do not need to rely on finite speed of propagation of the associated  wave operator  $\cos(t \mathcal L)$,  which  was extensively used to exploit a localization property of Bochner--Riesz operator since  the estimates for the kernel replace the role of finite speed of propagation.  Our approach also works for the Hermite Bochner--Riesz means, so it provides a simpler proof of the previous result in \cite{CDHLY21} (see Section \ref{sec:rem}).

\subsection*{Organization}
In Section \ref{sec:basic}, we prove two basic estimates which are to be used as main tools to prove Theorem \ref{thm:maxi}.
Section \ref{sec:max} and \ref{sec:sq} are devoted to the proof of Theorem \ref{thm:maxi}.
In Section \ref{sec:max}, we reduce the matters to obtaining a square function estimate with weights, which  we show  in Section \ref{sec:sq}.
In Section \ref{sec:nece}, we prove the necessity parts of  Theorem \ref{thm:ptconv} and  \ref{thm:aepoly}.
Finally, in Section \ref{sec:rem}, we make some remarks on a.e. convergence of the Hermite Bochner--Riesz means.

\subsection*{Notation}
For given positive numbers $A, B$,  $A\lesssim B$ means  $A\le C B$  for a constant $C>0$ depending only on $d$.
If the constant $C$ can be taken to be a number small enough, we use the notation $A\ll B$. 
Besides,  by $A\sim B$  we mean that  $A\lesssim B$ and $A\gtrsim B$. For $z\in \C^{d}$ and $M>0$,  $\mathbb B_d(z,M)$ denotes the  $2d$-dimensional ball in $ \C^{d}$ which is centered at $z$ and of radius $M$. For simplicity, we denote $\mathbb B_d(M)=\mathbb B_d(0,M)$.

\section{Preliminaries}\label{sec:basic}
In this section we obtain some estimates,  which we use to prove the main results.

\subsection{A local $L^2$  estimate for the spectral projection operator $\CP_\mu$}
We begin with an $L^2$ estimate for $\CP_\mu$ over balls centered at the origin.

\begin{lemma}\label{lem:trace}
 Let  $\mu\in 2\N_0+d.$ Then, there is a constant  $C$, independent of $M\ge1$ and $\mu$, such that
\begin{equation}\label{e:trace}
\int_{\mathbb B_d(M)} |\CP_\mu f(z)|^2 dz \le C M\mu^{-1/2} \|f\|_2^2.
\end{equation}
\end{lemma}

By using dyadic decomposition and Lemma \ref{lem:trace}  one can easily obtain the estimate
\begin{equation}\label{trace}
\int_{\C^{d}} |\CP_\mu f(z)|^2 
\Psi_\alpha(z)
dz \le C \mu^{-1/2} \|f\|_2^2 
\end{equation}
for a constant $C>0$ provided that  $\alpha>1$.  The estimate \eqref{trace} can be regarded as a trace lemma for $\mathcal L$ (cf.  
\cite[Lemma 1.5]{CDHLY21}).

To prove Lemma \ref{lem:trace} we modify  the argument in \cite{CDHLY21} to prove a trace lemma for $\mathcal H$.
We make use of the following two lemmas. 
Let  $\CL^a_k$ denote the normalized Laguerre function  of type $a$ which is given by
\[\CL^a_k (r^2/2) =\big({k!}/{(k+a)!}\big)^{1/2} \big({r^2}/2\big)^{a/2}L^a_k (r^2/2) e^{-r^2/4},\]
where  the Laguerre polynomial of type $ L^a_k$ is  defined by
\[ k! e^{-r} r^a L^a_k (r) =  ({d}/{dr)^k} (e^{-r }r^{k+a}). \]

\begin{lemma}[{\cite[Theorem 1.3.5]{Th93}}] \label{lagu} Let $\Phi_{a,b}$ denote a $1$-dimensional special Hermite function, $a,b \in \N_0$. For $z\in \C$,
we have
\[
\Phi_{a,b} (z) =
\begin{cases}
(2\pi )^{-1/2}  \big(\frac{-iz}{ |z|}\big)^{b-a} \CL^{b-a}_a (|z|^2/2),& \  a\le b,\\
(2\pi )^{-1/2} \big(\frac{i\bar z}{ |z|}\big)^{a-b} \CL^{a-b}_b (|z|^2/2),& \  a>b. 
\end{cases}
\]
\end{lemma}

\begin{lemma}[{\cite[Lemma 1.5.3]{Th93}}] \label{decay} Let $\ell =4k+2a +2$ and $a>-1$.
\[
|\CL^a_k(r)|\le C
\begin{cases}
(r\ell)^{a/2}, & 0\le r\le 1/\ell,\\
(r\ell)^{-1/4}, & 1/\ell\le r\le \ell/2,\\
\ell^{-1/4} (\ell^{1/3} +|\ell-r|)^{-1/4},& \ell/2\le r\le 3\ell/2,\\
e^{-\gamma r}, & r\ge 3\ell/2,
\end{cases}
\]
where $\gamma>0$ is a constant. Moreover, if $1\le r\le \ell -\ell^{1/3}$, we have
\begin{align}\label{eq:asympt}  
\CL^\alpha_k(r) =\frac{(2/\pi)^{\frac12}(-1)^k}{r^{\frac14}(\ell-r)^{\frac14}}\cos\Big( \frac{ \ell(2\theta -\sin 2\theta)-\pi}{4} \Big)+O\Big(\frac{\ell^{\frac14}}{(\ell-r)^{\frac74}}+(r\ell)^{-\frac34}\Big), 
\end{align}
where $\theta= \cos^{-1}(r^{1/2}\ell^{-1/2}).$
\end{lemma}

Note that $\Phi_{\sigma,\beta}(z)=\prod_{j=1}^d \Phi_{\sigma_j,\beta_j}(z_j)$ for  $\sigma=(\sigma_1,\cdots,\sigma_d),\beta=(\beta_1,\cdots,\beta_d)\in\mathbb N_0^d$, and $z=(z_1,\dots, z_d)\in \mathbb C^d$.  Using the estimates above, we obtain bounds for the special Hermite functions on $\C^d$.
\begin{proof}[Proof of Lemma \ref{lem:trace}]  We may write the projection operator $\CP_\mu$ as follows:
\[ \CP_\mu f =\sum_{j=1}^d \sum_{\sigma} \big(\sum_{\substack{\beta :\, 2|\beta| +d =\mu,  \beta_j\sim \mu}} \langle f, \Phi_{\sigma,\beta}\rangle \Phi_{\sigma,\beta}\big).\]
So, in order to show \eqref{e:trace}, we need only to prove 
\begin{equation}\label{tr_tem2}
\mathcal I_j:=\int_{\mathbb B_{d}(M)} |\sum_{\sigma} \sum_{\substack{\beta :\, 2|\beta| +d =\mu,  \beta_j\sim \mu}} \langle f, \Phi_{\sigma,\beta}\rangle \Phi_{\sigma,\beta}(z)|^2 dz \le C M\mu^{-1/2} \|f\|_2^2
 \end{equation}
for a constant  $C>0$.  By symmetry, it suffices to show \eqref{tr_tem2} with $j=1.$

Since $\mathbb B_d(M)\subset \mathbb B_1(M)\times \C^{d-1}$, setting $c^\sigma_{\beta} =\langle f, \Phi_{\sigma,\beta}\rangle,$ we see 
that  the left hand side of \eqref{tr_tem2} with $j=1$ is bounded above by 
\begin{align*}
&\sum_{\sigma,\sigma'\in \N_0^d} \sum_{\substack{2|\beta|+d=\mu;\\  \beta_1 \sim \mu}}\sum_{\substack{2|\beta'|+d=\mu;\\  \beta'_1 \sim \mu}} c^\sigma_{\beta}\,\overline{c^{\sigma'}_{\beta'}} \int_{\mathbb B_1(M)} \Phi_{\sigma_1,\beta_1}(z_1)\overline{\Phi_{\sigma_1',\beta_1'}(z_1)}\, dz_1 \prod_{l=2}^d\langle \Phi_{\sigma_l,\beta_l},\Phi_{\sigma_l',\beta_l'}\rangle\\
&=\sum_{\sigma_1,\sigma_1'\in\N_0} \sum_{\bar \sigma\in\N_0^{d-1}} 
  \sum_{\substack{\beta :\, 2|\beta| +d =\mu,  \beta_1\sim \mu}} c^{(\sigma_1,\bar\sigma)}_{\beta}\overline{c^{(\sigma_1',\bar\sigma)}_{\beta}} \int_{\mathbb B_1(M)} \Phi_{\sigma_1,\beta_1}(z_1)\overline{\Phi_{\sigma_1',\beta_1}(z_1)} dz_1.
\end{align*}
The  equality follows from orthogonality between  $\Phi_{\sigma_l,\beta_l}$.
We now claim that   
\begin{equation}\label{orth}
\int_{\mathbb B_1(M)} \Phi_{\sigma_1,\beta_1}(z_1)\overline{\Phi_{\sigma_1',\beta_1}(z_1)} dz_1  =0, \quad  \sigma_1\neq \sigma_1'.
\end{equation}
Assuming \eqref{orth} for the moment, we proceed to show \eqref{tr_tem2}. By \eqref{orth} it follows that 
\[
 \mathcal I_1\lesssim  \sum_{\sigma\in\Z^d} \sum_{\substack{\beta :\, 2|\beta| +d =\mu,  \beta_1\sim \mu}} |c^{\sigma}_\beta\,|^2 \int_{\mathbb B_1(M)} |\Phi_{\sigma_1,\beta_1}(z_1)|^2 dz_1. 
\]

Since $\sum_{\sigma, \beta} |c^\sigma_{\beta}|^2= \|f\|_2^2$,  we obtain  \eqref{tr_tem2} if we verify that 
\begin{equation}\label{using_decay}
\int_{\mathbb B_1(M)} | \Phi_{\sigma_1,\beta_1}(z_1)|^2 dz_1\le C  \mu^{-1/2}M
\end{equation}
with $C>0$ independent of $\sigma_1,\beta_1.$  
 If $M^2>\mu$, \eqref{using_decay} trivially holds  since 
  $\|\Phi_{\sigma_1,\beta_1}\|_2=1.$ Thus, to prove \eqref{using_decay}, we may assume $M^2\le \mu$. 
We first consider the case $\sigma_1>\beta_1.$ Then, we have $|\Phi_{\sigma_1,\beta_1} (z_1)| = (2\pi)^{-1/2} \CL^{\sigma_1-\beta_1}_{\beta_1}(|z_1|^2/2) $ by Lemma \ref{lagu}. Set $\epsilon=4\beta_1 +2(\sigma_1-\beta_1)+2,$ so $M^2 \lesssim \beta_1<\epsilon/2.$ 
Using the polar coordinates and Lemma \ref{decay} give
\begin{align*}
\int_{\mathbb B_1(M)} |\Phi_{\sigma_1,\beta_1}(z_1)|^2 dz_1& = \int_0^M |\CL^{\sigma_1-\beta_1}_{\beta_1} (r^2/2)|^2r dr\\
&\le C \int_0^{\sqrt{2/\epsilon}} (r^2\epsilon/2)^{\sigma_1-\beta_1} rdr +\int_{\sqrt{2/\epsilon}}^M (r^2\epsilon/2)^{-1/2} rdr\\
&\le C \epsilon^{-1}  +M\epsilon^{-1/2} \le C M \mu^{-1/2},
\end{align*}
since $\beta_1\sim \mu.$  When $\sigma_1\le \beta_1$, the estimate \eqref{using_decay}  can be proved in the same manner and we omit the detail.

It remains to prove \eqref{orth}. It is enough to show it for the case $\sigma_1>\sigma_1'$, since the other case can be handled by considering  complex conjugate. 
To make use of  Lemma \ref{lagu}, we need to consider the cases, $\sigma_1>\sigma_1'\ge \beta_1$, $\sigma_1> \beta_1\ge \sigma_1'$, and $\beta_1\ge \sigma_1>\sigma_1'$, separately. 
However, we only prove \eqref{orth} assuming  $\sigma_1>\sigma_1'\ge\beta_1$. The other cases can be similarly handled. 
Note that $({i\bar z_1}/{|z_1|}) \overline{({i\bar z_1}/{|z_1|})}=1.$ Hence, our assumption and Lemma \ref{lagu} give
\[ \Phi_{\sigma_1,\beta_1}(z_1)\overline{\Phi_{\sigma_1',\beta_1}(z_1)} = \frac1{2\pi} \Big(\frac{i \bar {z_1}}{|z_1|}\Big)^{\sigma_1-\sigma_1'} \CL^{\sigma_1-\beta_1}_{\beta_1} ({|z_1|^2}/2)  \CL^{\sigma_1'-\beta_1}_{\beta_1} ({|z_1|^2}/2). \]
 Using the polar coordinates again, we get
\begin{align*}
\int_{\mathbb B_1(M)}\!\! |\Phi_{\sigma_1,\beta_1}(z_1)|^2 dz_1
=\frac1{2\pi} \int_{0}^M \!\!\Big(\!\int_0^{2\pi} ({i e^{-i\theta}} )^{\sigma_1-\sigma_1'}  d\theta \Big) \CL^{\sigma_1-\beta_1}_{\beta_1} ({r^2}/2)  \CL^{\sigma_1'-\beta_1}_{\beta_1} ({r^2}/2) \,r d r,
\end{align*}
which  clearly vanishes because $\sigma_1-\sigma_1'\in\mathbb N.$ This completes the proof. 
\end{proof}

The following is  a slight  extension of Lemma \ref{lem:trace}, which is useful for the proof of the square function estimate (see Proposition \ref{prop:main2}).

\begin{coro}\label{lem:trace1}
Let $M\ge 1$ and  $\mu\in 2\N_0+d$. Suppose  that $\omega\in \mathrm C_c\big((\mu-\sigma, \mu+\sigma)\big)$ and $0< \sigma\le \mu$.  
Then we have
\begin{align}\label{e:gtra}
\|\chi_{\mathbb B_d(M)}\omega(\mathcal L)\|^2_{2\to 2}\le C\max{\{1, \sigma\}} M\mu^{-\frac12}\|\omega\|_\infty^2.
\end{align}
\end{coro}
\begin{proof}
If $\sigma\le 1$, then \eqref{e:gtra}  follows from Lemma \ref{lem:trace}.
So, we may assume $\sigma>1$.
By duality, \eqref{e:gtra} is equivalent to 
\[\|\omega(\mathcal L)\chi_{\mathbb B_d(M)}\|^2_{2\to 2}\le C\|\omega\|_\infty^2 \sigma M\mu^{-\frac12}.\]
For  $\ell =0,\cdots, \lfloor \sigma\rfloor +1$, we set ${\bf I}_\ell=[\mu-\sigma +2\ell,  \mu-\sigma+2\ell+2)$. There is only one eigenvalue of $\mathcal L$ in ${\bf I}_\ell$ for each $\ell$. So, orthogonality gives 
\begin{align*} 
\| \omega(\mathcal L) \chi_{\mathbb B_d(M)} \|^2_{2\to 2}
\le \!\! \sum_{\ell=0}^{\lfloor \sigma\rfloor +1}\!\|\chi_{{\bf I}_\ell}(\mathcal L)\omega(\mathcal L) \chi_{\mathbb B_d(M)}\|^2_{2\to 2}
\le \|\omega\|_\infty^2 \!\!\sum_{\ell=0}^{\lfloor \sigma\rfloor +1}\! \|\chi_{{\bf I}_\ell}(\mathcal L) \chi_{\mathbb B_d(M)}\|^2_{2\to 2}.
\end{align*}
By duality, the estimate  \eqref{e:trace} is equivalent to $\| \CP_{\mu'} \chi_{\mathbb B_d(M)} \|_{2\to 2}^2 \le C M(\mu')^{-1/2}$  for every $\mu'$. Applying this, we obtain 
\[ \|\omega(\mathcal L)\chi_{\mathbb B_d(M)}\|^2_{2\to 2} \le C \|\omega\|_\infty^2 \sum_{\ell=0}^{\lfloor \sigma\rfloor +1} M(\mu-\sigma +2l)^{-1/2} \le C \|\omega\|_\infty^2 \sigma M\mu^{-1/2}.\qedhere\]
\end{proof}

\subsection{ Kernel estimates}

In this section we consider estimates for the kernels of the operators $\eta((\mu-\mathcal L)/R)$ where $\mu\ge 1$, $R>0$, and $\eta\in \mathrm C_c^\infty((-2,2))$.
The following lemma shows that the kernels of these operators decay rapidly  from the diagonal $\{(z,z')\in \C^{2d}: z=z'\}$. 

For a given operator $T$, by $T(z,z')$ we denote the kernel of $T$. 

\begin{lemma}\label{lem:etapt}
Let $\mu\ge 1$, $R>0$, and $\eta\in \mathrm C_c^\infty ((-2,2))$. Then, 
 we have the following for any $N>0$ with the implicit constants depending only on $d$ and $N$. 
\begin{enumerate}
[leftmargin=1cm, labelsep=0.2 cm, topsep=-3pt]
    \item [$(i)\,\,$] Let $R\ge \mu$. If $|z-z'|\gtrsim R^{-1/2}$, then
    \begin{align}
    \label{1}
    |\eta((\mu-\mathcal L)/R)(z,z')|\lesssim R^d(1+R^{\frac12}|z-z'|)^{-N}. 
    \end{align}
    \item [$(ii)\,$]  Let $R\ge 1$ and $R<\mu$. If $ |z-z'|\gtrsim \mu^{1/2}R^{-1}$, then
    \begin{align}
     \label{2}     |\eta((\mu-\mathcal L)/R)(z,z')|\lesssim 
    R^d(1+R \mu^{-\frac12}|z-z'|)^{-N}.
\end{align}
\item[$(iii)$]  Let  $R<1$. If $|z-z'|\gtrsim \mu^{1/2}$,  then
\begin{align}
 \label{3}
     |\eta((\mu-\mathcal L)/R)(z,z')|\lesssim (1+\mu^{-\frac12}|z-z'|)^{-N}.
\end{align}
\end{enumerate}
\end{lemma}

To show Lemma \ref{lem:etapt}, we recall some properties of  the propagator $e^{-it\mathcal L}$. 
Since the eigenvalues of $\mathcal L$ are contained in $2\N_0+d$,  it follows that 
\begin{align}\label{id:period}
    e^{i(t+n\pi)\mathcal L} = (-1)^{nd}e^{it\mathcal L}, \quad n\in\Z.
\end{align}
The propagator $e^{-it\mathcal L}$ has an explicit kernel representation
\begin{align}\label{id:proparep}
e^{-it\mathcal L}(z,z') = C_d(\sin t)^{-d} e^{i\phi_{\mathcal L}(t,z,z')},
\end{align}
for a constant $C_d$ where 
\[\phi_{\mathcal L}(t,z,z') = \frac{|z-z'|^2\cos t}{4\sin t}+\frac12 \inp{z}{\mathbf S z'}.\]
This  can be shown from the kernel expression of the heat operator $e^{-t\mathcal L}$.
For the details, we refer to \cite[p.37]{Th93}.

\begin{proof}[Proof of Lemma \ref{lem:etapt}]
We first consider the case $R\ge 1$. By Fourier inversion
\Be
\label{hoho}
\eta\big((\mu-\mathcal L)/{R}\big) = \frac R{2\pi}\int \wh{\eta}(Rt) e^{it(\mu-\mathcal L)} dt.
\Ee 
Let $\eta_*\in \mathrm C_c^\infty \big((-\pi/2-2^{-3}, \pi/2+2^{-3})\big)$ such that $\sum_{n\in\Z}\eta_*(t + n\pi) = 1$ on $\R$.
We set 
\[ 
\eta_R(t) = (2\pi)^{-1}\sum_{n\in\Z} (-1)^{nd}e^{-in\pi\mu} \wh{\eta}\big(R(t - n\pi)\big)\eta_*(t).
\]
We write  
$
\eta\big((\mu-\mathcal L)/{R}\big) = \frac R{2\pi}\sum_{n\in\Z} \int \wh{\eta}(Rt)\eta_*(t + n\pi) e^{it(\mu-\mathcal L)} dt .
$
Changing  variables $t\to t-n\pi$ for each $n\in\Z$ and then using \eqref{id:period}, we get 
\begin{align}\label{ex:intop1}
  \eta\big((\mu-\mathcal L)/{R}\big) =  R\int \eta_R(t)  e^{it(\mu-\mathcal L)} dt.
\end{align}
It is easy to see that,  for any $N>0$, 
\begin{align}\label{i:etadvs}
    |\eta_R^{(k)}(t)|\lesssim R^k(1+R|t|)^{-N},\quad \forall  k\in\N_0. 
\end{align}

Combining \eqref{id:proparep} and \eqref{ex:intop1}, we get
\[
\eta\big((\mu-\mathcal L)/{R}\big)(z,z') = C_d R \int \eta_R(t) (\sin t)^{-d} e^{i(\mu t + \phi_{\mathcal L}(t,z,z'))} dt. 
\]
We  decompose the integral away from $t = 0$, at which the integrand has  singularity. 
Let $\psi \in \mathrm C^\infty_c([-8,-2]\cup[2,8])$ such that  $\sum_{j\in\Z}\psi(2^{j}t)=1$ for $t\neq0$ and set 
\[\eta_{R,j}(t)=\eta_R(t) \psi(2^j t)(\sin t)^{-d}.\]
Via a computation it is easy to show 
\begin{align}\label{i:kbdeta}
| (\eta_{R,j})^{(k)}(t)|\lesssim  2^{jd}(2^j + R)^k\big(1+ R 2^{-j}\big)^{-N},\quad k\in\N_0
\end{align}
for any $N>0$. It also follows that 
\begin{align}\label{ex:sumetar}
  \eta\big((\mu-\mathcal L)/{R}\big)(z,z') = \sum_{j\ge 1}K_j(z,z'),
  \end{align}
where 
\[K_j(z,z'):= C_d R \int  \eta_{R,j}(t) e^{i(\mu t + \phi_{\mathcal L}(t,z,z'))} dt.\]

We first show $(i)$, i.e.,  \eqref{1} under the assumption that $R\ge \mu$ and $ |z-z'|\gtrsim R^{-1/2}$. 
To this end, we set 
\[ \mathfrak I_1^a=  \sum_{j:2^{-j}\ll a^{-1/2}|z-z'| }K_j(z,z'),\quad \mathfrak I_2^a=\sum_{j:2^{-j}\gtrsim  a^{-1/2}|z-z'| }K_j(z,z') \]
 for $a>0$. So,  the sum in \eqref{ex:sumetar} equals $\mathfrak I_1^R+\mathfrak I_2^R.$ We consider  $\mathfrak I_1^R$ first. 
Note that 
\begin{align}\label{i:1dvphi}
{\partial_t \phi_{\mathcal L}(t,z,z') = -\frac{|z-z'|^2}{4\sin^2 t}.}
\end{align}
Since $2^{-j}\ll \mu^{-1/2}|z-z'|$,  so $|z-z'|^2/(4\sin^2t)\gg \mu$ if $t\in \supp{\psi(2^j\cdot)}$.
Thus, we have $ |\partial_t (\mu t+\phi_{\mathcal L}(t,z,z'))|\sim |z-z'|^2 2^{2j}$ for $t\in \supp{\psi(2^j\cdot)}$. 
Combining this and \eqref{i:kbdeta},  by repeated integration by parts we get
\begin{align}\label{i:bdetar1}
|K_j(z,z')|\lesssim R |z-z'|^{-2M}2^{(d-1-M)j}(1+R2^{-j})^{M-N}.
\end{align}
Choosing $N,M$ so that $d\ll M \ll N$ and combining the above inequality,  
we have
\begin{align}
\begin{aligned}\label{i:etai1}
  | \mathfrak I_1^R| &\lesssim R^{d-M}|z-z'|^{-2M}\lesssim R^d(1+R^{1/2}|z-z'|)^{-M},
\end{aligned}
\end{align}
provided that $|z-z'|\gtrsim R^{-1/2}$.
The estimate for $\mathfrak I_2^R$ is simpler.
By  \eqref{i:kbdeta}, we have 
\begin{align}
\label{see1}    | \mathfrak I^R_2|&\lesssim  R\int \bigg(\sum_{2^{-j}\gtrsim R^{-1/2}|z-z'|}  2^{jd} \psi(2^j t)\bigg) (1+R|t|)^{-N} dt \\
    &\lesssim   R^{\frac{d}2}|z-z'|^{-d}(1+R^{1/2}|z-z'|)^{-N+2}. \nonumber
\end{align}
Clearly, this implies $|\mathfrak I_2^R| \lesssim R^d(1+R^{1/2}|z-z'|)^{-N}$ if $|z-z'|\gtrsim R^{-1/2}$.
Combined with \eqref{i:etai1}, this yields the estimate in \textit{(i)}.

Now, we show $(ii)$, that is to say, \eqref{2}  assuming  $1\le R< \mu$ and $|z-z'|\gtrsim \mu^{1/2}R^{-1}$.
In this case, we split the sum \eqref{ex:sumetar} to get 
\[\eta\big((\mu-\mathcal L)/{R}\big)(z,z')=\mathfrak I_1^\mu (z,z')+\mathfrak I_2^\mu (z,z')
.\]
The proof follows the same argument as above.
When $2^{2j} |z-z'|^2\gg \mu$,
 it follows  from \eqref{i:1dvphi} that $|\partial_t(\mu t+ \phi_{\mathcal L}(t,z,z'))|\sim 2^{2j}|z-z'|^2$ for  $t\in\supp{{\psi}(2^j\cdot)}$.
Also, the bound \eqref{i:kbdeta} continues to hold for this case.
Now,  repeated integration by parts shows that \eqref{i:bdetar1} holds.
Choosing appropriate $N$ and $M$ and  taking sum  over  $j: 2^{-j}\ll \mu^{-1/2}|z-z'|$ give
\begin{align}
\begin{aligned}\label{i:etai2}
    |\mathfrak I_1^\mu| &\lesssim R \mu^{(d-1-{N})/2}|z-z'|^{-{N}-d+1},
   \end{aligned}
\end{align}
since $R<\mu$. Thus, we get $ |\mathfrak I_1^\mu |
    \lesssim R^d(1+R\mu^{-1/2}|z-z'|)^{-N-d+1}$ provided that $|z-z'|\gtrsim \mu^{1/2}R^{-1}$.
The estimate for $\mathfrak I_2^\mu$ can be obtained by applying an identical argument to that in the above.
Indeed, using $|e^{it\mathcal L}(z,z')|\lesssim t^{-d}$, we have
\[
|\mathfrak I^\mu_2|\lesssim R \int_{C\mu^{-1/2}|z-z'|}^2 (1+R|t|)^{-N}t^{-d}  dt \lesssim R^d(1+R\mu^{-1/2}|z-z'|)^{-N}.
\]
Combining  this and the estimate for $ \mathfrak I_1^\mu$   verifies \textit{(ii)}.

Finally, to show \textit{(iii)}, we assume that $R<1$ and $|z-z'|\gtrsim \mu^{1/2}$.  
Then, it is easy to see 
$|\eta_{R,j}^{(k)}(t)|\lesssim R^{-1} 2^{(d+k)j}$ for $k\in\N_0.$ By using \eqref{i:1dvphi} and repeated integration by parts as before,  we have
\[ |K_j(z,z')|\lesssim |z-z'|^{-2M} 2^{(d-1-M)j} \]
 for any $M$ if $2^{-j}\ll \mu^{-1/2}|z-z'|$. We decompose $\eta(R^{-1}(\mu-\mathcal L))(z,z')=\mathfrak I_1^\mu+\mathfrak I_2^\mu$. Thus, the above estimate  for  $ |K_j(z,z')|$  gives $|\mathfrak I_1^\mu| \lesssim  (1+\mu^{-1/2}|z-z'|)^{-M-d+1}.$
Since $|\eta_{R, j}|\lesssim |\eta_\ast\psi(2^j \cdot)|$, it follows that
$
|\mathfrak I^\mu_2|\lesssim R \int_{C\mu^{-1/2}|z-z'|}^2 |\eta_\ast(t)| dt \lesssim  (1+\mu^{-1/2}|z-z'|)^{-N}
$
for any $N$. This completes the proof. 
\end{proof}

\section{Proof of Theorem \ref{thm:maxi}}\label{sec:max}

In this section we reduce Theorem \ref{thm:maxi} to showing a  square function estimate. From now on, we identify $\C^d$ with $\R^{2d}$.

\subsection{Square function estimate}
We begin by recalling that 
\begin{align}\label{i:s*}
S_*^\delta(\mathcal L) f(z) \le C \sup_{R >0} \bigg(\frac{1}{R} \int_0^R |S_t^\rho(\mathcal L) f(z)|^2 dt\bigg)^{\frac12}
\end{align}
holds  for $\delta > \rho +1/2 > 0$. 
This was  shown in \cite[pp.278-279]{SW71} (see also \cite[p.13]{CDHLY21}).  
We  make a typical dyadic decomposition on the operator $S_t^\rho(\mathcal L)$.
Let $\phi_*\in \mathrm C^\infty_c((2^{-3}, 2^{-1}))$ {be a non-negative function} such that $\sum_{k\in\Z}\phi_*(2^k t) = 1$ for $t>0$.
For $k\ge 0$ we set 
\[
\phi_k(t)=\begin{cases} (2^kt)_+^\rho\,\phi_*(2^kt), & \quad k\ge 1, 
\\[3pt]
  \sum_{k\le 0} t_+^\rho\,\phi_*(2^k t),  &\quad k= 0. \end{cases}
\]
Decompose
\[
S_t^\rho(\mathcal L)f(z) = \phi_0\big(1-t^{-2}{\mathcal L}\big)f(z) + \sum_{k>0} 2^{-\rho k} \phi_k\big(1-t^{-2}{\mathcal L}\big)f(z).
\]
Substituting this into \eqref{i:s*} and then applying Minkowski's inequality, we get
\begin{align}\label{i:s*s0}
S_*^\delta(\mathcal L) f(z) \le \sum_{k\ge 0} 2^{-\rho k}  \mathfrak S_k f(z)
\end{align}
for $\rho > -1/2$ such that $\delta > \rho + 1/2$, where
\[
\mathfrak S_k f(z) := \sup_{R > 0} \Big(\frac{1}{R} \int_0^R \big|\phi_k\big(1-t^{-2}{\mathcal L}\big) f(z)\big|^2 dt\Big)^{\frac12}.
\]

The proof of Theorem \ref{thm:maxi} reduces to proving the next two propositions.

\begin{prop}\label{prop:main1}
Let $d\ge 1$ and $\alpha\ge 0$.  Then, we have  the estimate 
    \begin{align}\label{i:phi0}
        \int \sup_{R > 0}\big|\phi_0\big(1-R^{-2}{\mathcal L}\big)f(z)\big|^2 \Psi_\alpha(z) dz \le C \int |f(z)|^2 \Psi_\alpha(z) dz. 
    \end{align}
 \end{prop}

\begin{prop}\label{prop:main2} 
    Let  $d\ge 1$, $\alpha \ge 0$, and  $k\ge 1$. Then, for any $\ep > 0$  we have  
    \begin{align}\label{i:sk}
        \int \int \big|\phi_k \big(1-t^{-2}{\mathcal L}\big) f(z)\big|^2 \frac{dt}{t} \Psi_\alpha(z) dz \le C\,2^{(\ep-1)k} \mathbf B_\alpha(k)  \int |f(z)|^2 \Psi_\alpha(z) dz, 
    \end{align}
    where
\[ 
    \mathbf B_\alpha(k) = \begin{cases}
        \quad 1, &  \quad 0\le \alpha\le 1, \\
        2^{\frac{\alpha-1}{2}k}, & \quad \alpha > 1.
    \end{cases}\]
\end{prop}

Once we have the estimates \eqref{i:phi0} and \eqref{i:sk}, the proof of Theorem \ref{thm:maxi} is rather straightforward. 

\begin{proof}[Proof of Theorem \ref{thm:maxi}]
Choose an 
$\ep>0$ such  that $\delta > 2\ep + \max((\alpha-1)/4,0)$ and set  $\rho=\ep+ \max((\alpha-1)/4,0)-1/2$.  By \eqref{i:s*s0}, it suffices to show that the operator  
\[\textstyle \sum_{k\ge 0} 2^{-\rho k}  \mathfrak S_k\]  is bounded on $L^2(\R^{2d}, \Psi_\alpha)$.
Since $\mathfrak S_0 f\le \sup_R |\phi_0(1-R^{-2}\mathcal L)f|$, by \eqref{i:phi0} it follows that $
    \|\mathfrak S_0  f\|_{L^2(\R^{2d},\Psi_\alpha)}\lesssim \|f\|_{L^2(\R^{2d},\Psi_\alpha)}. $  For $k\ge 1$, it is clear that $(\mathfrak S_k f(z))^2\le \int |\phi_k (1-t^{-2}{\mathcal L}) f(z)|^2 \frac{dt}{t} $. Thus,  using \eqref{i:sk}  we obtain 
\[
 \| \sum_{k\ge 1} 2^{-\rho k}  \mathfrak S_kf  \|_{L^2(\Psi_\alpha)}  \lesssim   \sum_{k\ge 1} 2^{(\max(\frac{\alpha-1}4,0)-\frac12+ \frac\epsilon 2-\rho) k}   \|f\|_{L^2(\Psi_\alpha)}
 \lesssim \|f\|_{L^2(\Psi_\alpha)}.
\]
This completes the proof. 
\end{proof}

We prove Proposition \ref{prop:main1} for the rest of this section, while  Proposition \ref{prop:main2} is to be proved in  the next section.

\subsection{Proof of Proposition \ref{prop:main1}}
We start by showing  the  estimate
\begin{align}\label{i:ptbdp0}
\big|\phi_0 \big(1-R^{-2}{\mathcal L}\big) f(z)\big|\le C \int R^{2d}(1+R|z-z'|)^{-N} |f(z')| dz'
\end{align}
for $R\ge1$ and any $N\ge0$. 
To this end, it is enough to prove 
\begin{equation}\label{e0:sepa}
 \Big|\phi_0 \big(1-R^{-2}{\mathcal L}\big) (z,z')\Big| \le C R^{2d}(1+R|z-z'|)^{-N}.
 \end{equation}
 Recall that $\supp  \phi_0 \subset (2^{-3}, \infty)$ and let $\kappa$ be a smooth function on $\R$ such that $\supp\kappa\subset (-\infty,2)$ and $\kappa\equiv 1$ on $(-\infty,1]$.  Since  the eigenvalues of  $\mathcal L\ge 1$,  we have $\phi_0(1-R^{-2}\mathcal L) = (\phi_0\kappa)(1-R^{-2}\mathcal L)$ for $R>0$.
Applying the estimate \eqref{1} in Lemma \ref{lem:etapt} with $\eta= \phi_0\kappa$ and $R$, $\mu$ replaced by $R^2$ respectively, we get \eqref{e0:sepa} when $|z-z'|\gtrsim R^{-1}$. 
 For  the other case $|z-z'|\lesssim R^{-1}$, 
the required bound 
\eqref{e0:sepa} follows once we show
\[
\|\phi_0 \big(1-R^{-2}{\mathcal L}\big) \|_{L^1\to L^\infty}\lesssim R^{2d}. 
\]
This is an easy consequence of the estimate $\|\mathcal P_{\mu}\|_{L^1\to L^\infty}\le C \mu^{d-1}$, $\mu\in 2\N_0+d$ (see, for example, \cite{JLR22}). Indeed, this estimate and the triangle inequality give
$
\|\phi_0 \big(1-R^{-2}{\mathcal L}\big)\|_{L^1\to L^\infty}\le \textstyle  \sum_{\mu \in 2\N_0+d: \mu\le R^2}
\|\phi_0\|_\infty\|\mathcal P_{\mu}\|_{L^1\to L^\infty} \le C R^{2d}. 
$

We now proceed to prove \eqref{i:phi0}. Since the spectrum of $\mathcal L$ is contained in $[d,\infty)$, $\phi_0(1-R^{-2}\mathcal L) \equiv 0$ if $R\le 1$. So, the supremum in \eqref{i:phi0} can be replaced by supremum over $R\ge 1$. By replacing $f$ with $\Psi_\alpha^{-1/2} f$, the estimate   \eqref{i:phi0} is equivalent to
\[
    \int \sup_{R\ge 1}\big|\phi_0 \big(1-R^{-2}{\mathcal L}\big) (\Psi_\alpha^{-1/2} f)(z)\big|^2 \Psi_\alpha(z) dz \le C \int |f(z)|^2 dz. 
\]
Choose $N$ such that $N>\alpha + 100d$.
Let us set $\Phi(z) = (1+|z|)^{-N}$ and $\Phi_R(z)= R^{2d}\Phi(Rz)$.
By  \eqref{i:ptbdp0}, the matter is now reduced to showing 
\begin{align}\label{i:p01}
\int \sup_{R\ge 1} \big|\Phi_R \ast (\Psi_\alpha^{-1/2}|f|)(z)\big|^2\,\Psi_\alpha(z) dz \lesssim \int |f(z)|^2 dz.
\end{align}
Recall that $\Psi_\alpha = \sum_{j\ge 0}2^{-\alpha j}\chi_{\mathbb A_j}$  and set  
\begin{align*}
    \mathcal V_1 &= \sum_{j\ge 0} 2^{-\alpha j} \int_{\mathbb A_j} \sup_{R\ge 1} \Big|\Phi_R \ast \Big(\sum_{j' : |j-j'|\le 2} 2^{\frac{\alpha}{2}j'}\chi_{\mathbb A_{j'}}|f|\Big)(z)\Big|^2 dz, \\
    \mathcal V_2 &= \sum_{j\ge 0} 2^{-\alpha j} \int_{\mathbb A_j} \sup_{R\ge 1} \Big|\Phi_R \ast \Big(\sum_{j' : |j-j'|\ge 3} 2^{\frac{\alpha}{2}j'}\chi_{\mathbb A_{j'}}|f|\Big)(z)\Big|^2 dz.
\end{align*}
Clearly, the left hand side of \eqref{i:p01} is bounded by a constant times $\mathcal V_1 + \mathcal V_2$. 
Thus,  \eqref{i:p01} follows if we show 
\Be 
\label{i:p011}
\mathcal V_k  \lesssim \int |f(z)|^2 dz, \quad k=1,2. 
\Ee

We first  consider \eqref{i:p011} for $k=1$, which is easy to show. Indeed,  
\begin{align*}
\mathcal V_1 \lesssim   \sum_{|n|\le3}\sum_{j\ge 0} \, \int_{\mathbb A_j}\sup_{R\ge 1}(\Phi_R\ast(\chi_{\mathbb A_{j+n}}|f|)(z))^2 dz \lesssim \int  (\sup_{R>0}  \Phi_R\ast |f|(z))^2 dz.
\end{align*}
Since $\sup_{R>0}  \Phi_R\ast |f|(z)$ is  bounded by the Hardy-Littlewood maximal function, 
\eqref{i:p011} for $k=1$ follows.

We now consider \eqref{i:p011} for $k=2$.
Observe that 
\[
\mathcal V_2 \le \sum_{j\ge 0} 2^{-\alpha j} \int_{\mathbb A_j} \Big(\int \big(\sup_{R\ge 1} \Phi_R(z-z') \big) \big(\sum_{j' : |j-j'|\ge 3} 2^{\frac{\alpha}{2}j'}\chi_{\mathbb A_{j'}}|f|\big)(z')dz'\Big)^2 dz.
\]
Since $|j-j'|\ge 3$,  we have $|z-z'|\ge 1$ in the integral. Thus,  $\sup_{R\ge 1} \Phi_R(z-z') \le (1+|z-z'|)^{-{N+2d}}$.
Consequently, it follows that 
\[
\mathcal V_2 \lesssim  \int \Big(\int  \mathfrak K(z,z')  |  f(z')| dz' \Big)^2 dz,
\]
where
\[
\mathfrak K(z,z') = \chi_{\mathcal A}(z,z')(1+|z|)^{-\frac{\alpha}{2}}(1+|z'|)^{\frac{\alpha}{2}}(1+|z-z'|)^{-N+2d}
\]
and $\mathcal A=\{(z,z')\in\R^{2d}\times \R^{2d}: |z-z'|\ge 2^{-5} \max\{|z|,|z'|\}\}$.
Since $N>\alpha + 100d$, a simple calculation shows that 
$
\sup_z \int \mathfrak K (z,z') dz', \sup_{z'} \int \mathfrak K (z,z') dz <C
$
for a constant $C>0$.  By Young's inequality, \eqref{i:p011} for $k=2$ follows.

\section{Proof of square function estimates: Proof of Proposition \ref{prop:main2}}\label{sec:sq}
In the section, we show the square function estimate
\begin{align}\label{i:sq1}
  \mathfrak {T}_k  :=  
\int \int_1^\infty \Big|\phi_k \big(1-t^{-2}{\mathcal L}\big)(\Psi_{\alpha}^{-1/2} f)(z)\Big|^2 \Psi_\alpha(z) \frac{dtdz}{t} \lesssim 2^{(\ep-1)k} \mathbf B_\alpha(k)  \|f\|_2, 
\end{align}
which is equivalent to   \eqref{i:sk},  since  $\phi_k(1-t^{-2}\mathcal L)=0$ for $k\ge 1$ if $t\le 1$.  As to be seen below, the estimate \eqref{i:sq1} is easy to show  for $\alpha=0$  using orthogonality  (see the paragraph containing the inequality \eqref{e:hunif}). Consequently, by means of interpolation 
it is sufficient to show \eqref{i:sq1} for $\alpha>1$.

\subsection{Decomposition in $t$}
In order to prove \eqref{i:sq1}, we break the integral using the cutoff function $\phi_*$. So, we have  
\[
 \mathfrak {T}_k = \sum_{\nu\ge 2}\int \phi_*(2^{-\nu} t)\int \Big|\phi_k \big(1-t^{-2}{\mathcal L}\big)(\Psi_{\alpha}^{-1/2} f)(z)\Big|^2 \Psi_\alpha(z) dz\,\frac{dt}{t}.
\]
The spectral support of $\phi_k(1-t^{-2}\mathcal L)$ is contained in the interval $[t^2(1-2^{-k-1}), t^2(1-2^{-k-3})]$, which 
has length about $2^{2\nu}2^{-k}$ for  $t\in \supp{\phi_*(2^{-\nu}\cdot)}$. To exploit disjointness of spectral supports,  
we further break the integral in $t$  so that the spectral supports of the integrands are confined to  intervals of length about $2^{2\nu}2^{-k}$.

For the purpose, let  $\varphi\in \mathrm C_c^\infty((-3/4,3/4))$ such that  $\varphi\equiv 1$ on $[-1/4,1/4]$ and $\sum_{\ell\in\Z} \varphi(t-\ell) = 1$ for $t\in\R$. 
Thus,  $\mathfrak T_k$ is  equal to 
\[
\sum_{\nu\ge 2}\sum_{\ell\in\Z}\int \varphi\Big(\frac{t-2^{\nu-k}\,\ell}{2^{\nu-k}}\Big)\phi_*(2^{-\nu}t)\int \Big|\Psi_\alpha^{1/2}(z)\phi_k\big(1-t^{-2}{\mathcal L}\big)(\Psi_\alpha^{-1/2} f)(z)\Big|^2dz\,\frac{dt}{t}.
\]
Changing variables $t\to 2^{\nu-k}(t+\ell)$ yields 
\[
\mathfrak T_k= 2^{-k} \sum_{ \nu \ge 2}\,\sum_{\ell\in\Z} \int\varphi(t)\phi_*\Big(\frac{t+\ell}{2^{k}}\Big) \int |\Psi_\alpha^{1/2}(z)\phi_{\nu,\ell}(t, \mathcal L)(\Psi_\alpha^{-1/2} f)(z)|^2 dz\,\frac{2^{k}dt}{(t+\ell)},
\]
where \[
\phi_{\nu,\ell}(t,s) = \phi_k \Big(1-\frac{2^{2(k-\nu)} s}{(t+\ell)^2}\Big).
\]

We observe that the function $s\mapsto \phi_{\nu,\ell}(t,s)$ vanishes  for any $t\in\supp{\varphi}$ if $s$ is outside an interval of width $\sim 2^{2\nu}2^{-k}$, and the sets $\{\supp{\phi_{\nu,\ell}(t,\cdot)}\}_{\ell\in\Z}$ 
are boundedly overlapping. If the integral is nonzero,  $2^{-k}(t+\ell)\sim 1$ on $\supp \phi_*$. So, we have  
\[
\mathfrak T_k \lesssim 2^{-k}\,\sup_{t\in \supp{\varphi}}\, \sum_{ \nu \ge 2}\,\sum_{\ell\in L^k_{t} } \int |\Psi_\alpha^{1/2}(z)\phi_{\nu,\ell}(t, \mathcal L)(\Psi_\alpha^{-1/2} f)(z)|^2 dz,
\]
where 
\[ L^k_{t}=\{ {\ell}: 1/8\le 2^{-k}(t+ \ell)\le 1\}.\]
The estimate \eqref{i:sq1} for $\alpha=0$ follows since the supports of $\{ \phi_{\nu,\ell}(t, \cdot): \nu, \ell\}$ are boundedly overlapping. 
Hence,  the estimate \eqref{i:sq1} for $\alpha\in (0,1)$  follows by interpolation once we  have 
\begin{align}\label{e:hunif}
   \sum_{ \nu \ge 2}\, \sum_{\ell\in L^k_{t} }\int |\Psi_\alpha^{1/2}(z)\phi_{\nu,\ell}(t, \mathcal L)(\Psi_\alpha^{-1/2} f)(z)|^2 dz \le C2^{(\frac{\alpha-1}{2}+\ep)k} \|f\|_2, \quad \alpha>1
\end{align}
for a  constant $C>0$ whenever  $t\in\supp{\varphi}$.  

The rest of this section is devoted to showing \eqref{e:hunif}. 

\subsection{Low, middle, and high frequency parts}
To show \eqref{e:hunif}, we split the left hand side of \eqref{e:hunif} into three parts: 
\begin{align*}
    \mathfrak I^l (t) &:= \sum_{ 2\le\nu\le k/2}\, \sum_{\ell\in L^k_{t} }\int |\Psi_\alpha^{1/2}(z)\phi_{\nu,\ell}(t,\mathcal L)(\Psi_\alpha^{-1/2} f)(z)|^2 dz, 
    \\
    \mathfrak I^m  (t)  &:= \sum_{ k/2< \nu\le k}\, \sum_{\ell\in L^k_{t} }\int |\Psi_\alpha^{1/2}(z)\phi_{\nu,\ell}(t,\mathcal L)(\Psi_\alpha^{-1/2} f)(z)|^2 dz,
     \\
    \mathfrak I^h  (t)  &:= \sum_{ k< \nu}\, \sum_{\ell\in L^k_{t} }\int |\Psi_\alpha^{1/2}(z)\phi_{\nu,\ell}(t,\mathcal L)(\Psi_\alpha^{-1/2} f)(z)|^2 dz.
\end{align*}
We refer to  $\mathfrak I^l, \mathfrak I^m$, and $\mathfrak I^h$ as  the low, middle, and high frequency parts, respectively.
For those operators, we prove the next, from which  the desired estimate \eqref{e:hunif} follows immediately. 
\begin{prop}\label{prop:lmh}
    Let  $\alpha > 1$ and $t\in \supp \varphi$.  Suppose $ \|f\|_2\le 1$.   Then, for any $\epsilon>0$  there is a constant $C = C(d, \alpha, \ep)$ such that 
    \begin{align}\label{e:low}
        \mathfrak I^l(t) &\le C2^{(\frac{\alpha-1}{2}+\ep)k}, \\\label{e:mid}
        \mathfrak I^m(t) &\le C2^{(\frac{\alpha-1}{2}+\ep)k}, \\\label{e:high}
        \mathfrak I^h (t) &\le C2^{\ep k}.
    \end{align}
\end{prop}

Before proceeding to prove  Proposition \ref{prop:lmh},  we explain why we separately consider  $  \mathfrak I^l$, $  \mathfrak I^m$, and $  \mathfrak I^h$. 
Setting  $\mu=2^{-2k}(t+\ell)^2  2^{2\nu}$ and $R=2^{-2k}(t+\ell)^2 2^{2\nu-k}$, note that 
\Be 
\label{phinl}
 \phi_{\nu,\ell}(t,\mathcal L)= \phi\big((\mu-\mathcal L)/{R}\big).
 \Ee
 Since $2^{-k}(t+\ell)\sim 1$,  we have  $\mu\sim 2^{2\nu}$ and $R\sim 2^{2\nu-k}$.
 From \textit{(ii)} and \textit{(iii)} in Lemma \ref{lem:etapt}, we notice that the kernel of the operator $\phi_{\nu,\ell}(t,\mathcal L)$ changes its behavior around  $R = 1$, that is, $2^\nu= 2^{k/2}$. 
So, it is natural to distinguish  the low frequency part ($\nu\le k/2$) and the other part ($\nu> k/2)$.  Moreover, if  $\nu\ge k$,  Lemma \ref{lem:etapt} \textit{(ii)} shows that that  the kernel $\phi_{\nu,\ell}(t,\mathcal L)(z,z')$ is essentially localized to  $1$-neighborhood of the diagonal $\{(z,z'): z=z'\}$. In such a case, the weight $\Psi_\alpha^{1/2}$ can be handled easily. Thus, we additionally 
divide the part $\nu> k/2$ into the mid  frequency part ($k\ge \nu> k/2$) and high frequency part ($\nu> k$). 

\subsection{Proof of Proposition \ref{prop:lmh}}

We  show \eqref{e:low} first. 
For $j\ge 0$,   denote   
\begin{align*}
 U_j^l &= \{j'\in\N_0 :  \dist{(\mathbb A_j, \mathbb A_{j'})}
\le 2^{\nu+\ep k}\}.
\end{align*}
Recalling $\Psi^{1/2}_\alpha = \sum_{j\ge 0}2^{-\frac{ \alpha}{2}j}\chi_{\mathbb A_j}$, we decompose $\mathfrak I^l$ to get
\[
\mathfrak I^l \lesssim \mathfrak I^l_1 + \mathfrak I^l_2 + \mathfrak I^l_3,
\]
where
\begin{align*}
    \mathfrak I^l_1 &:= \sum_{ 2\le \nu \le k/2}\, \sum_{\ell\in L^k_{t} }\sum_{j\le  \nu+\ep k+3 }\int_{\mathbb A_j} \Big|\phi_{\nu,\ell}(t,\mathcal L)\Big(\sum_{j'\in  U_j^l}2^{\frac{\alpha}{2}(j'-j)}\chi_{\mathbb A_{j'}} f\Big)(z)\Big|^2 dz, \\
    \mathfrak I^l_2 &:= \sum_{ 2\le \nu \le k/2}\, \sum_{\ell\in L^k_{t} }\sum_{j>  \nu+\ep k+3 }\int_{\mathbb A_j} \Big|\phi_{\nu,\ell}(t,\mathcal L)\Big(\sum_{j'\in  U_j^l}2^{\frac{\alpha}{2}(j'-j)}\chi_{\mathbb A_{j'}} f\Big)(z)\Big|^2 dz, \\
    \mathfrak I^l_3 &:= \sum_{ 2\le \nu \le k/2}\, \sum_{\ell\in L^k_{t} }\sum_{j}\int_{\mathbb A_j} \Big|\phi_{\nu,\ell}(t,\mathcal L)\Big(\sum_{j'\notin  U_j^l}2^{\frac{\alpha}{2}(j'-j)}\chi_{\mathbb A_{j'}} f\Big)(z)\Big|^2 dz.
\end{align*}
Then, in order to show  \eqref{e:low}  it is sufficient to prove that  
\begin{align}\label{e:low1}
    \mathfrak I^l_1 &\le C2^{(\frac{\alpha-1}{2}+\ep)k} ,\\
    \label{e:low23}
    \mathfrak I^l_{{i}} &\le C, \qquad {i=2,3}
\end{align}
for some constant $C = C(d,\alpha,\ep)$. 
We first consider \eqref{e:low1}.
Since $s\to \phi_{\nu,\ell}(t,s)$ is supported in an interval of length $\lesssim 1$ (i.e., $R\lesssim 1$) which is centered at 
$\mu=2^{-2k}(t+\ell)^2  2^{2\nu}\sim 2^{2\nu} $,  using Corollary \ref{lem:trace1} with a suitable choice of $\omega$, we see that
\[
\int_{\mathbb A_j} |\phi_{\nu,\ell}(t,\mathcal L)
 g(z)|^2 dz
\lesssim  2^{-\nu} 2^j    \int |\phi_{\nu,\ell}(t,\mathcal L)g(z)|^2 dz.
\]
Thus, it follows that 
\begin{align}\label{e:il1}
    \mathfrak I^l_1 &\lesssim \sum_{ 2\le \nu \le k/2}\, 2^{-\nu} \, \sum_{j\le  \nu+\ep k+3} 2^{(1-\alpha) j}  \sum_{\ell\in L^k_{t} } \int |\phi_{\nu,\ell}(t,\mathcal L) F_j(z)|^2 dz,
\end{align}
where 
\[\textstyle F_j=\sum_{j'\in U^l_j}2^{\frac{\alpha}{2}j'}\chi_{\mathbb A_{j'}} f. \] 

Write  
$
 \sum_{\ell\in L^k_{t} } \int |\phi_{\nu,\ell}(t,\mathcal L) F_j(z)|^2 dz= \langle ( \sum_{\ell\in L^k_{t} }\phi^2_{\nu,\ell}(t,\mathcal L)) F_j, F_j\rangle.
$
Since $2^{2\nu-k}\le 1$, we have $ \sum_{\ell\in L^k_{t} }\phi^2_{\nu,\ell}(t,s) \lesssim 1$ for $s\ge 1$ and $t\in \supp{\varphi}$. Thus, it follows that 
\[ \sum_{\ell\in L^k_{t} } \int |\phi_{\nu,\ell}(t,\mathcal L) F_j(z)|^2 dz \le C\|F_j\|_2^2.\]
Combining this with \eqref{e:il1} and noting that  $\|F_j\|_2^2\lesssim  2^{\alpha\nu+\alpha \ep k }$, we get
\begin{align*}
\mathfrak I^l_1 &\lesssim \sum_{ 2\le \nu \le k/2} 2^{(\alpha - 1)\nu+\alpha \ep k } \sum_{j\le  \nu+\ep k+3} \!\! 2^{(1-\alpha)j}
\lesssim  \sum_{ 2\le \nu \le k/2}\!\! 2^{(\alpha - 1)\nu+\alpha \ep k }
\end{align*}
since $\alpha>1$. 
This  gives \eqref{e:low1}. 

We turn to the estimate for $\mathfrak I^l_2$. Observe that $ U_j^l \subset \{j+n: n=-1,0,1\} $ if $j>  \nu+\ep k+3$.
From this, we see that
\[
    \mathfrak I^l_2 \lesssim \sum_{|n|\le 1}\sum_{j\ge 0} \Big( \sum_{ 2\le \nu \le k/2}\, \sum_{\ell\in L^k_{t} } \int \big|\phi_{\nu,\ell}(t,\mathcal L)(\chi_{\mathbb A_{j+n}}f)(z)\big|^2 dz \Big)
    \]
As before, one can easily  check  that  $
\sum_{ 2\le \nu \le k/2}\, \sum_{\ell\in L^k_{t} }\phi^2_{\nu,\ell}(t,s)\le C
$
for any $s\ge 1$ and $t\in \supp{\varphi}$.  Hence, we get 
$\mathfrak I^l_2 \lesssim \sum_{|n|\le 1} \sum_{j\ge 0} \|\chi_{\mathbb A_{j+n}}f\|_2^2$. 
As a result, we get \eqref{e:low23} for {$i=2$.}

To complete the proof of \eqref{e:low}, it remains to show \eqref{e:low23} for {$i=3$}. By the Minkowski and H\"older inequalities, we have 
\[ \mathfrak I_3^l\lesssim \sum_{2\le \nu\le k/2}  \sum_{\ell\in L^k_{t} } \sum_j 2^{-\alpha j} \Big(\sum_{j'\notin U^l_j} 2^{\alpha j'/2} |\mathbb A_j|^{1/2} \|\chi_{\mathbb A_j} \phi_{\nu,\ell} (t,\mathcal L) \chi_{\mathbb A_{j'}} f\|_\infty \Big)^2.  \]
Note that $|z-z'|\ge \max\{ 2^{\nu+\ep k}, 2^j,2^{j'}\}$ for $(z,z')\in \mathbb A_j\times \mathbb A_{j'}$ since $j'\notin U^l_j$ for  $j\in \N.$  
Recalling \eqref{phinl},  we note that $\mu\sim 2^{2\nu}$ and $R\lesssim 1$. 
Choosing $N_1\gg1$ and $N_2, N_3>2d$, we make use of 
\eqref{3}  with $N=N_1+N_2+N_3$ to see 
$| \phi_{\nu,\ell} (t,\mathcal L)(z,z')| \lesssim  (1+2^{-\nu}|z-z'|)^{-N_1-N_2-N_3}$ for $(z,z')\in \mathbb A_j\times \mathbb A_{j'}$. Thus, we get  
 \begin{align*}
 |\phi_{\nu,\ell} (t,\mathcal L) (\chi_{\mathbb A_{j'}} f)(z)|
&\lesssim  2^{-\ep N_1k} (\max\{ 2^{j-\nu}, 2^{j'-\nu}\})^{-N_2}2^{d\nu}\|\chi_{\mathbb A_{j'}}f\|_2
\end{align*}
 for  $z\in \mathbb A_j$. 
 Combining this and the above inequality gives 
\begin{align*}
\mathfrak I_3^l
\lesssim 
\sum_{2\le \nu \le k/2} \sum_{\ell\in L^k_{t} }\sum_j  2^{-2\ep N_1k} 2^{-2N_2(j-\nu)} 2^{2d\nu}2^{2dj}
\lesssim  2^{(N_2+d+1-2\ep N_1)k} .
\end{align*}
Hence, the  inequality  yields  \eqref{e:low23} for {$i=3$} if we take $N_1$ large enough such  that $2\ep N_1 > 1+N_2+d$. 

Next we verify \eqref{e:mid} and \eqref{e:high}, of which  proofs  
follow the same line of argument as that of \eqref{e:low}. So, we shall be brief.
To show \eqref{e:mid}, we set
\begin{align*}
    U^m_j &:= \{j'\in\N_0 : \dist(\mathbb A_j,\mathbb A_{j'})\le 2^{(1+\ep)k-\nu}\},\quad j\in\N_0.
 \end{align*}
As before, denoting  $\tilde F_j= \sum_{j'\in  U^m_j}2^{\frac{\alpha}{2}j'}\chi_{\mathbb A_{j'}}f$ and 
\begin{align*}
    \mathfrak I^m_1 &= \sum_{ k/2\le \nu \le k}\, \sum_{\ell\in L^k_{t} }\sum_{ j\le (1+\ep)k-\nu+3}  2^{-\alpha j} \int_{\mathbb A_j} \Big|\phi_{\nu,\ell}(t,\mathcal L) \tilde F_j (z)\Big|^2 dz, \\
    \mathfrak I^m_2 &= \sum_{ k/2\le \nu \le k}\, \sum_{\ell\in L^k_{t} }\sum_{j> (1+\ep)k-\nu+3}  2^{-\alpha j} \int_{\mathbb A_j}  \Big|\phi_{\nu,\ell}(t,\mathcal L) \tilde F_j (z)\Big|^2 dz, \\
    \mathfrak I^m_3 &= \sum_{ k/2\le \nu \le k}\, \sum_{\ell\in L^k_{t} }\sum_{j}2^{-\alpha j} \int_{\mathbb A_j} \Big|\phi_{\nu,\ell}(t,\mathcal L)\Big(\sum_{j'\notin  U_j^m}2^{\frac{\alpha}{2}j'}\chi_{\mathbb A_{j'}} f\Big)(z)\Big|^2 dz,
\end{align*}
we have 
\[\mathfrak I^m \lesssim \mathfrak I^m_1 + \mathfrak I^m_2 + \mathfrak I^m_3.\]

We first  handle  $\mathfrak I^m_1$.   
As before, recalling \eqref{phinl} and applying Corollary \ref{lem:trace1} to the integral in  $\mathfrak I^m_1$, we obtain
\[ \mathfrak I^m_1 \lesssim \sum_{ k/2\le \nu\le k}\,2^{\nu-k}\, \sum_{ j\le (1+\ep)k-\nu+3}2^{-j(\alpha-1)}
\sum_{\ell\in L^k_{t} } \int |\phi_{\nu,\ell}(t,\mathcal L) \tilde F_j(z)|^2 dz. 
\]
Since $\supp \phi_{\nu,\ell}(t,\cdot)$ overlap at most $C$ times, $\sum_{\ell\in L^k_{t} } \int |\phi_{\nu,\ell}(t,\mathcal L) \tilde F_j(z)|^2 dz\lesssim \| {\tilde F_j}\|_2^2$. Thus, we get  
\begin{align*} 
\textstyle  \mathfrak I^m_1 \lesssim   \sum_{ k/2\le \nu\le k}\, 2^{(1-\alpha)\nu}2^{(\alpha-1+\ep\alpha)k}\lesssim 2^{(\frac{\alpha-1}{2} + \ep\alpha)k}
\end{align*}
because $\alpha>1$. Concerning $\mathfrak I^m_2$ and $\mathfrak I^m_3$,  we have the estimates  $\mathfrak I^m_2, \mathfrak I^m_3 \le C $, which 
 one can show  in the same way as  \eqref{e:low23}. More precisely,  the estimate  $\mathfrak I^m_2\le C$ can be obtained similarly as  $\eqref{e:low23}$ for {$i=2$}. 
Likewise,  to show $\mathfrak I^m_3 \le C$, using \eqref{2}  instead of \eqref{3},   one can repeat  the argument which shows \eqref{e:low23} for {$i=3$}.  
 We omit the details.  Combining those estimates for  $\mathfrak I^m_2, \mathfrak I^m_2,$ and $\mathfrak I^m_3$ gives \eqref{e:mid}. 

Finally, to show \eqref{e:high}, 
we break $\mathfrak I^h$ into two parts to have $\mathfrak I^h\lesssim \mathfrak I^h_1+\mathfrak I^h_2$, where 
\begin{align*}
    \mathfrak I^h_1 &:= \sum_{  \nu> k}\, \sum_{\ell\in L^k_{t} }\sum_{j}\int_{\mathbb A_j} \Big|\phi_{\nu,\ell}(t,\mathcal L)\Big(\sum_{j'\in  U_j^m}2^{\frac{\alpha}{2}(j'-j)}\chi_{\mathbb A_{j'}} f\Big)(z)\Big|^2 dz, \\
    \mathfrak I^h_2 &:= \sum_{ \nu > k}\, \sum_{\ell\in L^k_{t} }\sum_{j}\int_{\mathbb A_j} \Big|\phi_{\nu,\ell}(t,\mathcal L)\Big(\sum_{j'\notin  U_j^m}2^{\frac{\alpha}{2}(j'-j)}\chi_{\mathbb A_{j'}} f\Big)(z)\Big|^2 dz.
\end{align*}
Thus, it is sufficient to show that
\[
\mathfrak I^h_1 \le C2^{\ep k},\quad \mathfrak I^h_2 \le C
\]
for $t\in\supp\varphi.$  The second inequality can be obtained in the same manner as the estimate $\mathfrak I^m_3\le C$ shown.
 So, we only prove  the first estimate. 
Since $\nu> k$, $ U_j^m\subset \{j\pm n : n =0,1,\cdots, n_0\}$ for  a positive integer  $n_0$  such that $2^{n_0}\sim 2^{\ep k}$. Thus,
\[ \mathfrak I^h_1 \lesssim {2^{(\alpha+1)\ep k}}\sum_{|n|\le n_0} \sum_{j\ge 0} 
    \Big( \sum_{\nu>k}\, \sum_{\ell\in L^k_{t} }\int \big|\phi_{\nu,\ell}(t,\mathcal L)(\chi_{\mathbb A_{j+n}}f)(z)\big|^2 dz\Big).\]
    Note that  $\sum_{\nu>k}  \sum_{\ell\in L^k_{t} }\phi_{\nu,\ell}^2(t,s)\le C$ for $s\ge 1$ and $t\in \supp{\varphi}$. So,  
    the expression inside the parenthesis is bounded above by  $C \|\chi_{\mathbb A_{j+n}}f\|_2^2$. Thus, we obtain 
 \begin{align*}
   \textstyle \mathfrak I^h_1 \lesssim 2^{(\alpha+1)\ep k}  \sum_{|n|\le n_0} \sum_{j\ge 0}  \|\chi_{\mathbb A_{j+n}}f\|_2^2 \lesssim C{2^{(\alpha+2)\ep k}}
\end{align*}
as desired.

\section{Sharpness of  summability indices}\label{sec:nece}
In this section we discuss sharpness of summability indices 
given  in Theorem \ref{thm:ptconv} and Corollary \ref{thm:aepoly}.  
The following proves the necessity parts of  Theorem \ref{thm:ptconv} and Corollary \ref{thm:aepoly}.

\begin{prop}\label{prop:nece} Let $d\ge1$, $\beta\ge 0$, and   {$p>4d/(2d-1+2\beta)$}. If $0\le \delta< {\gamma(p, 2d,\beta)/2}$, then there exists a measurable function  $f$ such that 
$ \Psi_\beta f\in L^p(\C^d)$ and  \eqref{eq:mea} holds. 
\end{prop}

To prove Proposition \ref{prop:nece}, we construct a sequence of functions that behave as if they were  the eigenfunctions of $\mathcal L$ on the set $ \mathbb A_1$. 

\begin{lemma}\label{lem:const} Let $p>4d/(2d-1)$and  $\beta\ge 0$. 
Then, there are sequences $ \{ \mu_k\}\subset 2\N_0 +d$ and $\{f_k\} \subset \mathcal S(\mathbb C^d)$ such that  
\[ 
\mu_k\sim 2^{2^k},\quad \|\Psi_{\beta} f_k\|_{{L^p(\C^d)}}=1,\]
and the following hold
 for a large constant  $k_\circ:$
\Be
\Big |\Big \{ z\in\mathbb A_1: | \CP_{\mu_k} f_k(z)|\ge C_0\mu_k^{{ \gamma(p,2d,\beta)/2}}\Big\}\Big|\ge C_0
\label{c_lower}
\Ee
for a constant $C_0>0$ if  $k\ge k_\circ$, and  for any $N>0$ there is  a constant $ C_N>0$  such that 
\Be
|\CP_{\mu_k}f_j(z)|\le C_N 
{\mu_k^{\gamma(p,2d,\beta)/2}(\mu_j/\mu_k)^{\frac\beta2+\frac12-\frac dp} }
|\mu_k-\mu_j|^{-N}, \quad z\in \mathbb A_1  \label{c_upper}
\Ee
 whenever $j\neq k\ge {k_\circ}$.
\end{lemma}

Assuming  Lemma \ref{lem:const}  for the moment, we prove Proposition \ref{prop:nece}.

\begin{proof}[Proof of Proposition \ref{prop:nece}] 
Let  $ \{ \mu_k\}$ and $\{f_k\}$ be the sequences given in Lemma \ref{lem:const}. 
We consider  $f=\sum_{k= k_\circ}^\infty 2^{-k} f_k$ and  
\[ E_k =\Big \{ z\in \mathbb A_1 :  | S^\delta_*(\mathcal L) f(z)|\ge c\,2^{-k}\mu_k^{-\delta +{ \gamma(p,2d,\beta)/2}} \Big\},\quad k\ge k_\circ
\]
for  a small positive constant $c$ to be chosen later.  For  \eqref{eq:mea}, it is enough to show
\begin{equation}\label{ek_lo}
|E_k|\ge C_0 
\end{equation}
for a constant $C_0>0$ if  $k\ge {k_\circ}$. Indeed, since $\mu_k\sim 2^{2^k}$ and  {$\delta<  \gamma(p, 2d,\beta)/2$}, it is easy to see  that $\{E_k\}$ is a decreasing sequence of measurable sets which converges  to $E:=\{z\in \mathbb A_1 : S^\delta_*(\mathcal L)f(z)=\infty\}$. 
 Thus,   \eqref{eq:mea} follows from \eqref{ek_lo}.

The inequality  \eqref{ek_lo} is an easy consequence of  \eqref{c_lower} and  the inclusion relation 
\begin{equation}\label{inclusion}
\tilde E_k:=  \Big\{ z\in \mathbb A_1: |\CP_{\mu_k}f_k(z)|\ge C_0\mu_k^{{\gamma(p,2d,\beta)/2}}\Big\} \subset E_k.
\end{equation} 
Hence, it is enough to show \eqref{inclusion}. To this end,  we invoke the inequality
\begin{equation}\label{eq:relation}
|\CP_{\mu_k} f(z)|\le \mathfrak C \mu_k^\delta S^\delta_{\ast}(\CL) f(z),\quad k\ge k_\circ,
\end{equation}
which holds with a constant $\mathfrak C$. This follows from the well-known identity
\[F(\mathcal L)=\frac1{\Gamma(\delta+1)}\int_0^\infty F^{(\delta+1)}(t)t^\delta S^\delta_t (\mathcal L) dt,\quad \delta\ge0 \]
for $F\in \mathrm C_c^\infty ([0,\infty))$ 
where    $F^{(\delta)}$ denotes the Weyl fractional derivative of $F$ and $\Gamma(\delta)$ is the  gamma function. Substituting $F=\eta(\cdot -\mu_k)$ for $\eta\in \mathrm C^\infty_c((-1,1))$, we get \eqref{eq:relation}. See \cite{{CDHLY21},LR22} for the detail. 

From \eqref{eq:relation}, we now have
\begin{align}
\label{hh}
S^\delta_{\ast}(\CL) f(z)
\ge {\mathfrak C}^{-1} \mu_k^{-\delta} \Big( 2^{-k}|\CP_{\mu_k}f_k(z)| -\sum_{j\neq k} 2^{-j}|\CP_{\mu_k}f_j(z)|\Big).
\end{align}
Using \eqref{c_upper} with a sufficiently large $N$, we see  that $\sum_{j\neq k} 2^{-j}|\CP_{\mu_k}f_j(z)|$ is bounded above by  a constant times 
\begin{align*}
 \mu_k^{{\gamma(p,2d,\beta)}/2} \sum_{j\neq k}2^{-j} (\mu_j/\mu_k)^{\frac12-\frac dp+\frac\beta2} |\mu_j-\mu_k|^{-N} \lesssim  \mu_k^{{\gamma(p,2d,\beta)/2}} \mu_k^{-1}.
\end{align*}
We  choose a constant  $c$ such that $c< C_0/2\mathfrak C $.
Using \eqref{hh}, for $z\in \tilde E_k$
we have $S^\delta_*(\mathcal L)f(z)\ge c\,2^{-k}\mu_k^{-\delta+{\gamma(p,2d,\beta)/2}}$ if $k$ is large enough. 
Thus,   \eqref{inclusion}  follows.
\end{proof}

We now turn to prove Lemma \ref{lem:const}.

\begin{proof}[Proof of Lemma \ref{lem:const}] 
We take a sequence   $\{\mu_k\}\subset 2\N_0+d$ such that $\mu_k\sim 2^{2^k}$. Set 
\[ g_k (z) =\phi_*^\vee (\CL-\mu_k)(0,z),\quad z\in\mathbb C^d,\]
{where $\phi_*\in \mathrm C^\infty_c((2^{-3}, 2^{-1}))$ defined in Section \ref{sec:max}. 
}  
From \eqref{hoho} and \eqref{id:proparep} we have 
\Be \label{lg} g_k(z) =\frac{C_d}{2\pi} \int \phi_*(t) (\sin t)^{-d} e^{-i(\frac{|z|^2}4 \cot t +\mu_k t)} dt.\Ee
It is easy to see that 
\begin{equation}\label{eq_c_bound}
|g_k(z)|\le 
\begin{cases}
\qquad  \  \ C\mu_k^{-1/2}, & \text{if}\ |z|^2\sim \mu_k,\\
C_N(1+\max\{|z|^2, \mu_k\})^{-N}, &\text{otherwise}
\end{cases}
\end{equation}
for every $N\in\N_0$. 
Indeed,  the phase function $p(t): = -\mu_k t -|z|^2 \cot t/4 $ satisfies  $|p'(t)|\gtrsim \max\{\mu_k,\, |z|^2\}$ for  $t\in \supp \phi_*$  if  $|z|^2\ge C\mu_k$ or $|z|^2\le C^{-1}\mu_k$ for a constant $C>0$. 
So, integration by parts gives $|g_k(z)|\lesssim (1+\max\{\mu_k,\, |z|^2\})^{-N}$ unless $|z|^2\sim \mu_k.$  If $|z|^2 \sim \mu_k$,   
$p''(t)\sim \mu_k$
for  $t\in (2^{-3}, 2^{-1})$.  The stationary phase method  gives  $|g_k(z)|\sim \mu_k^{-1/2}$  (for example, see \cite{St93}). For $k\ge {k_\circ}$ large enough,  
\eqref{eq_c_bound} gives 
\begin{equation}\label{Lp}
\|\Psi_{\beta}g_k\|_{{L^p(\C^d)}}\sim \mu_k^{{d}/p-1/2-\beta/2}.
\end{equation}

Note that 
\Be \label{ho}  
\CP_{\mu_k} g_j(z) = \phi_*^\vee(\mu_k-\mu_j)\CP_{\mu_k}(0,z), 
\Ee
and recall that the kernel of $\CP_{\mu_k}$ is given by 
\begin{equation}\label{kernel_Pk}
 \CP_{\mu_k}(w,z)=\frac 1{(2\pi)^{d}} \Big(\frac{(N_k+d-1)!}{N_k!}\Big)^{\frac12} \Big(\frac{|w-z|^2}2\Big)^{-\frac{d-1}2}\mathcal L^{d-1}_{N_k}\Big(\frac{|w-z|^2}2\Big) e^{\frac i2 \langle w, {\mathbf S}{z}\rangle}
 \end{equation}
with $2N_k+d =\mu_k$ (\cite[Ch. 1--2]{Th93}). 
Using  \eqref{eq:asympt} for $\alpha=d-1$, we see that 
\Be 
\label{haha}
\Big|\Big\{ z\in \mathbb A_1:  |\CP_{\mu_k}(0,z)|\sim \mu_k^{{(2d-3)}/4} \Big\}\Big|\ge C_0
\Ee
for a constant $C_0>0$ and for $k\ge {k_\circ}$ large enough (see, for example, \cite[Proof of Lemma 4.9]{CDHLY21}).  We set 
\[ f_k =g_k/\| \Psi_{\beta} g_k\|_{{L^p(\C^d)}}.\]

It remains to verify \eqref{c_lower} and \eqref{c_upper}.
 In fact,   \eqref{c_lower} follows from \eqref{haha} and \eqref{Lp} since  {$\phi_*^\vee(0)>0$} and  $\CP_{\mu_k} f_j=\phi_*^\vee(\mu_k-\mu_j) \CP_{\mu_k}(0,z)/\|\Psi_{\beta}g_j\|_{{L^p(\C^d)}}.$ Using  \eqref{eq:asympt} for $\alpha=d-1$, one can easily see that $|\CP_{\mu_k}(0,z)|\lesssim \mu_k^{{(2d-3)}/4}$ 
if $|z|\sim 1$. Combining this and  \eqref{Lp},  we have  
\[
 |\CP_{\mu_k} f_j(z)|\lesssim {\mu_k^{\gamma(p,2d,\beta)/2}(\mu_j/\mu_k)^{\frac\beta2+\frac12-\frac dp} }|\phi_*^\vee(\mu_k-\mu_j)|, \quad |z|\sim 1
\]
 for $k, j$ large enough. By rapid decay of $\phi_*^\vee$, this gives \eqref{c_upper}.
\end{proof}

\begin{rem}
\label{l2w} Using $g_k$ in the proof of Lemma \ref{lem:const}, one can easily show that \eqref{e:maxi} fails if $\delta<(\alpha-1)/4$. Indeed, 
making use of \eqref{eq:relation},  \eqref{ho}, \eqref{haha} and  \eqref{Lp}, we see that the estimate \eqref{e:maxi} implies 
\[ \mu_k^{{(2d-3)}/4- \delta}\lesssim  \mu_k^{{d}/2-1/2-\alpha/4}. \]
Taking $k\to \infty$ gives $\delta\ge (\alpha-1)/4$. 
\end{rem}

\section {Bochner--Riesz means for the  Hermite operator}\label{sec:rem}

The operators $\mathcal L$ and $\mathcal H$ have common spectral properties such as periodicity of the associated propagators and  spectrums bounded  away from the zero.
Furthermore we have a similar kernel representation of the operator $\eta((\mu-\mathcal H)/{R})$ as before (cf. \eqref{ex:intop1}) using the propagator $e^{it \mathcal H}$, whose kernel is given by 
\begin{align}\label{id:propher}
e^{-i\frac{t}2\mathcal H}(x,y) = \tilde C_d(\sin t)^{-d/2} e^{i\phi_{\mathcal H}(t,x,y)} 
\end{align}
for a constant $\tilde C_d$ (\cite{Th93, JLR222})  where
\begin{equation}\label{phase_H}
\phi_{\mathcal H}(t,x,y) = 2^{-1}(|x|^2+|y|^2)\cot t-\inp{x}{y}\csc t.
\end{equation}
Making use of those properties, it is not difficult to see that our approach also works for the Bochner--Riesz means for the  Hermite operator. 
In fact, one can prove the following which is different from the result in \cite{CDHLY21} in that 
no upper bound  is required on $\alpha$.

\begin{thm}\label{thm:maxih}
Let $\alpha > 0$ and $\psi_\alpha= (1+|x|)^{-\alpha}$. If $\delta>\max\{(\alpha-1)/4, 0\}$, then for a constant $C>0$ we have the estimate 
\[
\|S_*^\delta(\mathcal H) f \|_{L^2(\R^d, \psi_\alpha)} \le C \|f\|_{L^2(\R^d, \psi_\alpha)}.
\]
\end{thm}

Consequentially,  we obtain a.e. convergence for functions with growth at infinity.

\begin{coro}
Let  $d\ge 1$ and $\beta\ge 0$. If $\delta > {\gamma(p,d,\beta)/2}$, then  
$
\lim_{t\to\infty} S_t^\delta(\mathcal H) f= f$ {a.e.} 
 whenever ${\psi_\beta} f \in L^p{(\R^d)}$. Conversely, 
   if  a.e. convergence holds for all $f$ satisfying $ {\psi_\beta} f\in L^p({\R^d})$ for some $p\in ({2d/(d-1+2\beta)},  \infty]$, then $\delta \ge {\gamma(p,d,\beta)/2}$.
\end{coro}

\subsection*{Necessity part} The necessity  part can be shown in the same manner as in Section \ref{sec:nece}.
Recalling the cutoff function $\phi_*$ defined in Section \ref{sec:max}, we consider
\[
\tilde g_k(x) := \phi_*^\vee(\mu_k - \mathcal H)(0,x), \quad k\in \N, 
\]
 where $\mu_k =2 N_k+d$  is a sequence  such that $\mu_k\sim 2^{2^k}$ and $N_k$ is even.  Set $\tilde f_k=\tilde g_k/\| {\psi_\beta} \tilde g_k\|_p$ and $ \tilde f= \sum_k 2^{-k} \tilde f_k$. 
Then, following the same argument as in Section \ref{sec:nece}, one can see without difficulty that  
\[
\big|\{x\in \R^d : \sup_{t}|S_t^\delta(\mathcal H) \tilde f(x)= \infty\}\big|\gtrsim 1
\]
provided that $\delta < {\gamma(p,d,\beta)}/2$. Indeed, using the Fourier inversion, we write
\Be
\label{kernelh}
\eta\big((\mu-\mathcal H)/{R}\big)f(x) = \frac R{4\pi}\int \wh{\eta}\big({Rt}/2\big) e^{i\frac{t}2(\mu-\mathcal H)}f(x) dt
\Ee
for $\eta\in \mathrm C_c^\infty ((-2,2))$ and $f\in\mathcal S(\R^d)$ (cf. \eqref{hoho}).  From \eqref{id:propher} and \eqref{phase_H} we note that
 \[\tilde g_k(x) ={\tilde C_d} \int \phi_*(t) (\sin t)^{-\frac d2} e^{-i(\frac{|x|^2}4 \cot t +\mu_k t)} dt\] for a constant ${\tilde C_d}$. 
Thus, similarly as before,  we have $\|\psi_{\beta}\tilde g_k\|_{{L^p(\R^d)}}\sim \mu_k^{{d}/(2p)-1/2-\beta/2}.$  

Let $\tilde \CP_\mu$ denote the spectral projection (associated with the Hermite operator  $\mathcal H$ in $\mathbb R^d$) to the space spanned the eigenfunctions of the eigenvalue $\mu$. Then, it follows 
that $\tilde \CP_{\mu_k} \tilde g_j(x) = \phi_*^\vee(\mu_k-\mu_j)\tilde \CP_{\mu_k}(0,x)$. We now claim that 
\begin{equation}\label{e_hermite}
|\tilde\CP_{\mu_k}(0,x)| \sim \Big(\frac{\Gamma (N_k/2+d/2)}{\Gamma(N_k/2+1)}\Big)^{1/2}|x|^{1-\frac d2} |\mathcal L^{d/2-1}_{N_k/2}(|x|^2)|.
\end{equation}
Once we have this, the subsequent  argument  is identical to that of the case of the twisted Laplacian $\mathcal L$. So,  we omit the detail. 

Finally, we verify \eqref{e_hermite}. 
When $d$ is even, \eqref{e_hermite} is immediate from \eqref{lg} and 
\eqref{kernel_Pk}. However, for odd $d$, we need some additional work. 
%
Using \eqref{id:propher}, \eqref{phase_H}, and the fact that $\mu_k\in 2\N_0+d$, we have 
\[\tilde\CP_{\mu_k}(0,x) = c_d \int_0^{2\pi}  (\sin t)^{-d/2} e^{-i( \mu_kt + \frac{|x|^2}2\cot t)}  dt\] 
 (see \cite{JLR222}). Recall  that $N_k$ is chosen to be even. Since $\tilde\CP_{\mu_k}(0,x) $ is radial and since  $\tilde\CP_{\mu_k}(\tilde\CP_{\mu_k}(0,x))=\tilde\CP_{\mu_k}(0,x)$,  by  \cite[Corollary 3.4.1]{Th93}
 we have
\begin{equation}\label{e:proj_ra}
 \tilde\CP_{\mu_k}(0,x) = \mathfrak C_k^d\,  \Big(\frac{\Gamma(N_k/2 +d/2)}{\Gamma(N_k/2+1)}\Big)^{1/2}  |x|^{1-\frac d2}\,\mathcal L^{d/2-1}_{N_k/2}(|x|^2)
 \end{equation}
 for some constant $\mathfrak C_k^d$.  Thus,   \eqref{e_hermite} follows if we  show $|\mathfrak C_k^d|\sim 1$. Writing 
  $\tilde \CP_{\mu_k}(0,x)=\sum_{|\alpha|=N_k}\Phi_\alpha(0)\Phi_\alpha(x)$, by orthogonality of the Hermite functions  we see
\[ \textstyle \| \tilde\CP_{\mu_k}(0,\cdot)\|_2 = \big(\sum_{|\alpha|=N_k} |\Phi_\alpha(0)|^2\,\big)^{1/2} \sim \mu_k^{\frac{d-2}4}\] 
because  $|\Phi_\alpha(0)|\sim \mu_k^{-d/4}$ for  most $\alpha$. Besides, using the polar coordinate and the estimates for the normalized Laguerre functions \cite[(i) in Lemma 1.5.4]{Th93}, one can easily see that  the $L^2$ norm of the right-hand side of  \eqref{e:proj_ra} is comparable to 
\[\textstyle |\mathfrak C^d_k|  \, \mu_k^{\frac {d-2}4}  \left( \int_0^\infty  | \mathcal L^{d/2-1}_{N_k/2}(r^2)|^2 r dr\right)^{1/2} \sim |\mathfrak C^d_k|  \, \mu_k^{\frac {d-2}4}. \]
Therefore, we have $|\mathfrak C_k^d|\sim 1$.

\subsection*{Sufficiency part}  One can prove Theorem \ref{thm:maxih}  in the same manner as  Theorem \ref{thm:maxi}. In fact, it is clear that 
we only need to verify that 
Corollary \ref{lem:trace1}  and Lemma \ref{lem:etapt} hold while  $\mathcal L$  and $\mathbb C^d$ replaced by $\mathcal H$ and $\mathbb R^d$, respectively. 
Since  Corollary \ref{lem:trace1}  follows from  the estimate \eqref{e:trace}, so does the desired estimate for the Hermite operator from the estimate (2.3) in \cite{CDHLY21}.
As for the counterpart of Lemma \ref{lem:etapt}, we have the following.

\begin{lemma}\label{lem:etapth}
Let $\mu\ge 1$, $R>0$, and $\eta\in \mathrm C_c^\infty ((-2,2))$. Then, 
 we have the following for any $N>0$ with the implicit constants depending only on $d, N$. 
\begin{enumerate}
[leftmargin=1cm, labelsep=0.2 cm, topsep=-3pt]
    \item [$(i)'\,\,$] Let $R\ge 1$ and $R\ge \mu$. If $|x-y|\gtrsim R^{-1/2}$, then
    \[
    |\eta((\mu-\mathcal H)/R)(x,y)|\lesssim R^{d/2}(1+R^{1/2}|x-y|)^{-N}. 
    \]
    \item [$(ii)'\,$]  Let $R\ge 1$ and $R<\mu$. If $ |x-y|\gtrsim \mu^{1/2}R^{-1}$, then
    \[
       |\eta((\mu-\mathcal H)/R)(x,y)|\lesssim 
    R^{d/2}(1+R \mu^{-1/2}|x-y|)^{-N}.
\]
\item[$(iii)'$]  Let  $R<1$. If $|x-y|\gtrsim \mu^{1/2}$,  then
\[
     |\eta((\mu-\mathcal H)/R)(x,y)|\lesssim (1+\mu^{-\frac12}|x-y|)^{-N}.
\]
\end{enumerate}
\end{lemma}

\begin{proof}  The proof is similar to that of Lemma \ref{lem:etapt}. So, we shall be brief.  
Recall    $\eta_*$ in the proof of Lemma \ref{lem:etapt}   which satisfies $\sum_{n\in\Z}\eta_*(t + n\pi) = 1$.  Combining this with \eqref{kernelh} and changing variables $t\to t+n\pi$,  we write
\[\eta\big((\mu-\mathcal H)/{R}\big)f = \frac R{4\pi}\sum_{n\in\Z}\int \wh{\eta} \big({R(t+n\pi)}/2\big)\eta_*(t) e^{i\frac{t+n\pi}2(\mu-\mathcal H)}f dt.\]
Note that  $\phi_\mathcal H(t+n\pi,x,y)=\phi_\mathcal H(t,x,(-1)^n y)$  for $n\in\Z$. Thus, we have 
\[\eta((\mu-\mathcal H)/R)(x,y) =  \sum_{\tilde n=0,1} R\int {\eta}_{R}^{\tilde n}(t) e^{i(\frac {t\mu}2+ \phi_\mathcal H(t,x,(-1)^{\tilde n} y))} dt, \]
where
\begin{align*}
\eta_R^{\tilde n}(t) = \frac{\tilde C_d}{4\pi}\sum_{n\in\Z} e^{-i\mu\pi(2n+\tilde n)/2}\, \frac{\wh{\eta}\big(R(t - (2n+\tilde n)\pi)/2\big)}{( (-1)^{\tilde n}\sin t)^{d/2}}\, \eta_*(t), \quad \tilde n=0,1. 
\end{align*}
It is easy to see that   the estimate \eqref{i:etadvs} with   $\eta_R$  replaced by  $\eta_R^{\tilde n}$ holds. As before, we 
set 
$ \eta_{R,j}^{\tilde n}= {\psi}(2^j \cdot) \eta_R^{\tilde n}$, 
which clearly satisfies 
\Be
\label{hhh}
| (\eta_{R,j}^{\tilde n})^{(k)}(t)|\lesssim  2^{jd/2}(2^j + R)^k\big(1+ R 2^{-j}\big)^{-N},\quad k\in\N_0
\Ee
for any $N>0$.  We  dyadically 
decompose  the kernel {$\eta\big((\mu-\mathcal H)/{R}\big)(x,y) $} to get 
  \[
\eta\big((\mu-\mathcal H)/{R}\big)(x,y) =\sum_{j, \tilde n} K_j^{\tilde n}(x,y):= \sum_{\tilde n=0,1} \sum_{j\ge 1} R \int  \eta_{R,j}^{\tilde n}(t) e^{i(\mu t/2 + \phi_{\mathcal H}(t,x,(-1)^{\tilde n} y)))} dt. 
\] 
To show $(i)'$, we split the sum above as follows: 
\[ \sum_{\tilde n=0,1}  \Big( \mathfrak I_1^{\tilde n, R}+\mathfrak I_2^{\tilde n, R}\Big):=  \sum_{\tilde n=0,1}  \Big(\sum_{j:2^{-j}\ll R^{-1/2}|z-z'| }K_j^{\tilde n} +\sum_{j:2^{-j}\gtrsim  R^{-1/2}|z-z'| }K_j^{\tilde n}\Big).\]
For $(ii)'$ and $(iii)'$, we decompose $\eta\big((\mu-\mathcal H)/{R}\big)(x,y) =\sum_{\tilde n=0,1} ( \mathfrak I_1^{\tilde n, \mu}+\mathfrak I_2^{\tilde n, \mu}) $.
It is rather straightforward  to show the desired estimates for $\mathfrak I_2^{\tilde n, R}$ and $\mathfrak I_2^{\tilde n, \mu}$, for which we do not need to use 
the oscillatory effect of the kernel (see, for example, \eqref{see1}).

To show the estimates for 
$\mathfrak I_1^{\tilde n, R}$ and $\mathfrak I_1^{\tilde n, \mu}$, we need to exploit oscillatory effect. 
However, one can complete the proof without difficulty, following the same lines of argument in the {\it proof of Lemma  \ref{lem:etapt}} once we show
\begin{equation}\label{lower}
 |\partial_t \big({\mu t}/2 + \phi_\mathcal H(t,x, (-1)^{\tilde n} y)\big)| \gtrsim 2^{2j}|x-y|^2
 \end{equation}
for  $t\in \supp (\eta_{R,j}^{\tilde n})$ and $\tilde n=0,1,$ 
provided that $2^{-j}\ll \mu^{-1/2}|x-y|.$ Indeed, this combined with \eqref{hhh} gives, via integration by parts,  
\[ 
|K_j^{\tilde n} (x,y)|\lesssim R |x-y|^{-2M}2^{(d-1-M)j}(1+R2^{-j})^{M-N}.
\] 
Consequently, all the desired estimates follow in the same manner as before.   

 Finally, we verify the estimate \eqref{lower}. Note that 
\[
 -\partial_t\phi_\mathcal H(t,x,(-1)^{\tilde n}y) = \frac{|x-y|^2}{2\sin^2 t} +\frac{\inp{x}{y}}{\sin^2 t} (1-(-1)^{\tilde n} \cos t).
 \]
We distinguish the two cases $\inp{x}{y}\ge0$ and $\inp{x}{y}<0$. For the first case, we have $|\partial_t\phi_\mathcal H(t,x,y)|\ge {|x-y|^2}/(2\sin^2 t)$. 
For the latter, note that $|x-y|^2\ge 2|\inp{x}{y}|$, so $|\partial_t\phi_\mathcal H(t,x,y)|\ge {|x-y|^2} \cos t/(2\sin^2 t)$. Thus, \eqref{lower} follows for $t\in \supp (\eta_{R,j}^{\tilde n})$ and $\tilde n=0,1$. 
As a result,  we get \eqref{lower} since $2^{-j}\ll \mu^{-1/2}|x-y|$.
\end{proof}

\section*{Acknowledgements} This work was supported by  NRF (Republic
of Korea) grants  2020R1F1A1A0-1048520 (Jeong), 2022R1A4A1018904 (Lee), and  KIAS Individual Grant MG087001 at Korea Institute for Advanced
Study (Ryu).

\end{document}